\documentclass[oneside,reqno,12pt]{amsart} 

\newcount \version
\newcount \shortversion
\newcount \longversion
\version \longversion\relax

\usepackage{amsmath,amsthm,amscd}
\usepackage{amssymb,amsfonts} 
\usepackage{graphicx}
\usepackage{epstopdf}
\usepackage{fullpage}
\usepackage{caption}
\usepackage{subcaption}

\font\Bbb=msbm10
\def\BBB#1{\hbox{\Bbb#1}}

\theoremstyle{definition}
\newtheorem{definition}{Definition}[section]
\newtheorem{example}[definition]{Example}
\newtheorem{remark}[definition]{Remark}
\theoremstyle{plain}
\newtheorem{theorem}[definition]{Theorem}
\newtheorem{proposition}[definition]{Proposition}
\newtheorem{corollary}[definition]{Corollary}
\newtheorem{lemma}[definition]{Lemma}

\newcommand\cost{\hbox{\rm Cost}}
\newcommand\CC{{\mathcal C}}
\newcommand\M{{\mathcal M}}
\newcommand\HH{{\mathbb H}}
\newcommand\R{{\mathbb R}}

\newcommand\Ham{{\mathcal H}}

\newcommand\htx{{\widehat{t}_X}}
\newcommand\hty{{\widehat{t}_Y}}
\newcommand\hsx{{\widehat{s}_X}}
\newcommand\hsy{{\widehat{s}_Y}}

\newcommand\oo{{o}}
\newcommand\ad{\hbox{\rm ad}}

\newcommand\mybox{
\hbox{
{\kern -8pt}
{\rule [-2pt]{1pt}{13pt}}
{\kern -8pt}
{\rule [-2pt]{14pt}{1pt}}
{\kern -22pt}
{\rule [10pt]{14pt}{1pt}}
{\kern -8pt}
{\rule [-2pt]{1pt}{13pt}}
}}

\makeatletter
\newcommand{\LeftEqNo}{\let\veqno\@@leqno}
\makeatother

\def\bU{\overline{U}}

\def\H{{\BBB H}}
\def\ep{\varepsilon}
\def\px{x}
\def\ad{\hbox{\rm ad}}

\date{}
\begin{document}

\title
[Optimal attitude control]
{Optimal attitude control with two rotation axes}
\author{Yuly Billig}
\address{School of Mathematics and Statistics, Carleton University, Ottawa, Canada}
\email{billig@math.carleton.ca}

%\centerline
%{\bf Optimal attitude control with two rotation axes.}

\

\

\

\begin{abstract}
Euler proved that every rotation of a 3-dimensional body can be realized as a sequence of three
rotations around two given axes. If we allow sequences of an arbitrary length, such a decomposition
will not be unique. In this paper we solve an optimal control problem minimizing the total angle
of rotation for such sequences. We determine the list of possible optimal patterns that give
a decomposition of an arbitrary rotation. Our results may be applied to the attitude control of 
a spacecraft with two available axes of rotation.
\end{abstract}

\maketitle

\section{Introduction}

 In this paper we investigate the problem of optimal attitude control of a 3-dimensional body which
can be rotated around two fixed axes. The problem goes back to Euler \cite{E} who proved in 1776
that an arbitrary rotation $g$ of a 3-dimensional body may be factored as
\begin{equation}
\label{Eu}
g = R(t_1 Y) R(t_2 X) R(t_3 Y),
\end{equation}
where $R(t X)$ (resp. $R(tY)$) is a rotation in angle $t$ around $X$-axis (resp. $Y$-axis). The parameters $t_1, t_2, t_3$
are called the Euler's angles.

 We could allow decompositions for $g$ with more factors:
\begin{equation}
\label{Eutwo}
g = R(t_1 Y) R(t_2 X) \ldots R(t_{n-1} Y) R(t_n X).
\end{equation}
Clearly we will get infinitely many such decompositions for a fixed element $g$ in the group $SO(3)$ of rotations.
Thus it is natural to pose the question of finding a decomposition (\ref{Eutwo}) that minimizes the total angle
of rotation $|t_1| + |t_2| + \ldots + |t_n|$. It can happen that decompositions with more factors have a smaller 
total angle of rotation than the Euler's decomposition (\ref{Eu}).

It turns out that this problem is not well-posed: for some $g$ an optimal decomposition (\ref{Eutwo}) does not
exist. Instead, the infimum of the total angle of rotation is attained as a limit on a sequence of decompositions
(\ref{Eutwo}) with $n \to\infty$.

We can overcome this difficulty by noting that
\begin{equation*}
R (aX + bY) = \lim\limits_{n\to\infty} \left( R \left( \frac{a}{n} X\right) R\left(\frac{b}{n} Y\right) \right)^n, 
\end{equation*}
hence it is natural to extend the set of controls from $\{ \pm X, \pm Y \}$ to
$\CC = $ \break
$\left\{ aX + bY \ \big| \ |a| + |b| = 1 \right\}$. Implementing a rotation $R(aX + bY)$ corresponds to
carrying out rotations around axes $X$ and $Y$ simultaneously with the ratio $a:b$ of angular velocities.

Once we extend the control set, our optimization problem becomes well-posed and for every $g \in SO(3)$
there is an optimal decomposition with a finite number of factors.

We study this problem in a more general setting, where we allow an arbitrary angle $0< \alpha \leq \frac{\pi}{2}$ between the
axes $X$ and $Y$. We also introduce a more general cost function to be minimized, where a rotation in angle $t$ around
$Y$-axis has the same cost as a rotation around $X$-axis in angle $\kappa t$, $0 \leq \kappa \leq 1$.

We solve the optimization problem in this greater generality and determine possible patterns for the optimal 
decompositions. Each of these patterns has (at most) 3 independent time parameters, and it is fairly easy
to find  numerically the decompositions of a given element $g \in SO(3)$ according to each pattern. This
produces a finite number of decompositions and we can immediately see which one of them is optimal.

It happens that our optimization problem has a bifurcation at $\kappa = \cos \alpha$. For the cases
$\kappa > \cos \alpha$ and $\kappa < \cos \alpha$ we get different lists of optimal patterns. There are
also special cases when $\kappa = 0$ or $\cos \alpha = 0$.

Let us present the list of optimal patterns in case when the axes $X$ and $Y$ are perpendicular to each other and $\kappa = 1$.
Since in this case the problem is symmetric with respect to the dihedral group of order $8$, generated by transformations
$(X, Y) \mapsto (Y, X)$, $(X, Y) \mapsto (-X, Y)$, $(X, Y) \mapsto (X, -Y)$, the list of patterns will also be symmetric with respect
to this group. We denote this group of $8$ symmetries by $(X, Y) \mapsto \{ \pm X, \pm Y \}$ and use it to present the list of
patterns in a more compact form.

\begin{theorem}
Let the angle between the axes $X$ and $Y$ be $\alpha = \frac{\pi}{2}$ and let $\kappa = 1$.
For an element $g \in SO(3)$ there is an optimal decomposition with $t_1, t_2, t_3 \geq 0$ of one of the following types:
\begin{align*}
& R(t_1 X) R(t_2 Y) R(-t_3 X), \hspace{1.9cm} \text{with \ } t_1, t_3 \leq t_2 \leq \pi, \\
& R(t_1 X) R(t_2 Y) R(-t_2 X) R(-t_3 Y), \hspace{0.3cm} \text{with \ } t_1, t_3 \leq t_2 \leq \pi, \\
& R(t_1 X) R\left( t_2 \left( {X+Y} \right)/2 \right) R(t_3 X), \hspace{0.4cm} \text{with \ } t_1, t_3 \leq \pi, \  t_2 \leq \sqrt{2} \pi,  \\
& R(t_1 X) R\left( t_2 \left( {X+Y} \right)/2 \right) R(t_3 Y), \hspace{0.4cm} \text{with \ }  t_1, t_3 \leq \pi, \  t_2 \leq \sqrt{2} \pi,  
\end{align*}
and symmetric to these under the group of transformations $(X, Y) \mapsto \{ \pm X, \pm Y \}$.
\end{theorem}

\

\begin{example}
Suppose we would like to decompose a rotation $R(t Z)$ as a product of rotations around $X$- and $Y$-axes, where $\{ X, Y, Z\}$ is the 
standard orthogonal basis of $\R^3$. The pattern for the optimal decompositions will depend on the value of $t$. 
If $0 \leq t \leq  \frac{\pi}{2}$ then the following decomposition realizes the minimum of the total rotation angle:
\begin{multline*}
\hfill
R (t Z) = R( -t_1 X) R ( -t_2 Y) R( t_2 X) R( t_1 Y), \hfill \\
\text{where} \quad \quad
 t_1 = \arccos \left( \frac{1}{ \cos \left( \frac{t}{2} \right) + \sin  \left( \frac{t}{2} \right)} \right), \
 t_2 = \arccos \left( \cos \left( \frac{t}{2} \right) - \sin  \left( \frac{t}{2} \right) \right).
\hfill
\end{multline*}
For $ \frac{\pi}{2} \leq t \leq  \pi$ the Euler's decomposition (\ref{Eu}) becomes optimal:
\begin{equation*}
R (t Z) = R( - \frac{\pi}{2} Y) R ( t X) R(\frac{\pi}{2} Y).
\end{equation*}
When $t = \frac{\pi}{2}$ both patterns are optimal. For $-\pi < t < 0$ the optimal decompositions may be obtained by 
switching $X$ with $Y$ in the above expressions.
\end{example}

In case when the axes $X$ and $Y$ are perpendicular to each other and $\kappa = 0$ 
(meaning that rotations around $Y$-axis have zero cost), the optimal decompositions are precisely those
described by Euler (\ref{Eu}).

In 2009 NASA launched a space telescope Kepler with a mission of finding planets outside the Solar system. 
This spacecraft was placed in an orbit around the Sun. To take images of stars, the telescope needs to be pointed
in the target direction, with its solar panels facing the Sun. The attitude control of Kepler is done with reaction wheels,
which are heavy disks mounted on electric motors. Once the reaction wheel is turned, the spacecraft will turn around the
same axis in the opposite direction due to the angular momentum conservation law. 

If we have three reaction wheels  with linearly independent axes, by rotating them simultaneously with appropriate relative 
angular velocities, we can implement a continuous rotation of the spacecraft around an arbitrary axis.  For redundancy,
Kepler was equipped with four reaction wheels with their axes in a tetrahedral configuration, so that any three of them
could provide an efficient attitude control. However by May 2013, two of the four reaction wheels failed, leaving 
Kepler with just two available axes of rotation \cite{NASA}. The results of our paper provide optimal methods
for attitude control with two rotation axes, like in situation with the Kepler space telescope.

 This paper builds on our previous work \cite{B}, where we studied a similar problem for $SU(2)$, also with two available controls,
but with a restriction that only a positive time evolution is allowed. That  paper was motivated by the applications to quantum
control in a 1-qubit system.

 In the present paper we use the geometric control theory \cite{J}, which is an adaptation of the Pontryagin's Maximum
Principle to the setting of Lie groups. The Maximum Principle provides only necessary conditions for optimality, which need
not be sufficient. In Section 3 we identify decompositions that satisfy the necessary conditions of the Pontryagin's Maximum Principle.
Then we go into a more detailed analysis in Section 4 by showing that decompositions with a large number of factors are
not optimal, even when they satisfy the conditions of the Maximum Principle. 
% Since the Maximum Principle is in a way a local
% maximum test, we have to use either non-local algebraic relations in $SO(3)$, or use the second order analysis. 
Our main results are stated in Theorems  \ref{czero} -- \ref{Meq} at the end of the next Section. 

\

{\bf Acknowledgements.} I thank Cornelius Dennehy, Ken Lebsock, Eric Stoneking and Alex Teutsch for the stimulating discussions.  
Support from the Natural Sciences and Engineering Research Council of Canada is gratefully acknowledged.

\

\section{Attitude control problem}

 For a unit vector $X \in \R^3$ denote by $R(tX)$ an operator of rotation of $\R^3$ in angle $t$ around $X$, with the plane perpendicular to $X$ turning counterclockwise when viewed from the endpoint of $X$. As a $3\times 3$ matrix, $R(tX)$ is given
by 
% the Rodrigues' formula:
the formula:
\begin{equation*}
R(tX) =  \cos (t) I +  \sin (t) \ad X +  (1 - \cos (t))  X X^T ,
\end{equation*}  
where for $ X = (a, b, c)^T$ with $a^2 + b^2 + c^2 = 1$,  $\ad X$
is the adjoint matrix of $X$ with respect to the cross product, so that $ (\ad X) Y = X \times Y$:
\begin{equation}
\label{cross}
 \ad X  = \left( \begin{matrix} 0 & -c  & b \\ c & 0 & -a \\ -b & a & 0 \\ \end{matrix} \right) 
\quad
\text{ and }
\quad
X X^T = \left( \begin{matrix} a^2 & ab & ac \\ ab & b^2 & bc \\ ac & bc & c^2 \\
\end{matrix} \right) .
\end{equation}

For $X \in \R^3$ with $\left| X \right| \neq 1$ we set $R(tX) = R(t^\prime X^\prime)$, where $t^\prime = t \left| X \right|$ 
and $X^\prime = {X}/{\left| X \right|}$. The set of all rotations of $\R^3$ forms the group $SO(3)$.

For a fixed $X$ the set $\left\{ R(tX) | t \in \R \right\}$ is a 1-parametric subgroup in $SO(3)$.

It is well-known that for any two non-proportional unit vectors $X, Y$, the corresponding $1$-parametric subgroups together generate the whole group $SO(3)$. This means that every element $g \in SO(3)$ may be decomposed into a product
\begin{equation}
\label{dec}
g = R(t_1 C_1) R(t_2 C_2) \ldots R(t_n C_n)
\end{equation}
with $C_j \in \left\{ X, Y \right\}$. Decomposition (\ref{dec}) is of course not unique. It is then natural to consider the optimization
problem of finding the infimum of $\left| t_1 \right| + \ldots + \left| t_n \right|$ over all decompositions (\ref{dec}) with fixed
$g \in SO(3)$. More generally, we may assign cost to each generator $X, Y$ and minimize the total cost in (\ref{dec}). 

Introduction of the cost parameters may be warranted in case when the body that we control has unequal momenta of inertia with respect to the axes $X$ and $Y$, 
thus making it easier to rotate it around one of the axes.

 Without loss of generality, we assume that $\cost(X) \geq \cost(Y)$ and renormalize the cost function by fixing $\cost(X) = 1$,
$\cost(Y) = \kappa$ with $0 \leq \kappa \leq 1$.

 For the rest of the paper we fix two non-proportional vectors $X, Y \in \R^3$ with $\left| X \right| = \left| Y \right| =1$. An important parameter is the angle $\alpha$ between these vectors. Without loss of generality we assume $0 < \alpha \leq \frac{\pi}{2}$, otherwise we can replace $Y$ with $-Y$. Throughout the paper we will the use parameter $c = \cos(\alpha)$, $0 \leq c < 1$. Let $Z$ be a vector perpendicular to $X$ and $Y$, $Z = X \times Y$, $\left| Z \right| = \sin (\alpha)$.

  It could happen that the infimum of cost is not attained on any particular decomposition (\ref{dec}), but rather as a limit on a sequence of such decompositions with $n \to \infty$. It turns out that we can overcome this difficulty by enlarging the set of generators to be
\begin{equation*}
 \CC = \left\{ a X + b Y \big| \left| a \right| + \left| b \right| = 1 \right\} .
\end{equation*}

Note that rotations corresponding to elements of $\CC$ can be realized as limits of products of 
rotations with axes $\{ X, Y \}$:
\begin{equation}
\label{explim} 
R(t (aX + bY)) = \lim_{n \to \infty} \left( R\left( \frac{ta}{n} X\right)  R\left( \frac{tb}{n} Y\right) \right)^n . 
\end{equation}
From the point of view of the attitude control, this corresponds to turning on controls $X$ and $Y$ simultaneously with intensities $a$ and $b$ respectively.

We extend the definition of the cost function in such a way that the cost of both sides
of (\ref{explim}) is the same:
\begin{equation}
\label{costf}
\cost (a X + b Y) = \left| a \right| \cost(X) + \left| b \right| \cost(Y).
\end{equation}

 Our goal is to solve the following

{\bf Problem 1.} For a given $g \in SO(3)$ find a decomposition
$g = R(t_1 C_1) R(t_2 C_2) \ldots R(t_n C_n)$
with $C_1, \ldots, C_n \in \CC$, $t_1, \ldots, t_n \geq 0$,
realizing the the infimum of
$$t_1 \cost (C_1) + \ldots + t_n \cost(C_n). $$

 It was shown in \cite{B}, Theorem 1.4, that the infimum cost in this problem problem is the same as for its more restricted version where the set of controls is taken to be 
$\left\{ \pm X, \pm Y \right\}$ instead of $\CC$.
  
 In fact, we shall see that we would not need the whole set $\CC$, but require in addition to controls  $\left\{ \pm X, \pm Y \right\}$ only the elements  $\left\{ \pm W_+, \pm W_- \right\}$, where
$W_+$ (resp. $W_-$) is a linear combination of $X$ and $Y$, which is orthogonal to $\kappa X +Y$
(resp.  $\kappa X - Y$). 

% Note that $\cost (\kappa X) = \cost (Y)$, so all the controls on a segment
% connecting $\kappa X$ with $\pm Y$ have the same cost.

% It is clear that the restriction  $\left| a \right| + \left| b \right| = 1$ is not really essential and we could allow more general controls $a X + b Y$ with appropriate cost, 
%\begin{equation}
%\label{costf}
%\cost (aX + bY) = \left| a \right| + \kappa \left| b \right|.
%\end{equation}

Since the cost of $R((\lambda t) C)$ and $R(t (\lambda C))$ is the same, we can rescale the generators without changing the cost of decompositions.
We can thus drop the requirement $ \left| a \right| + \left| b \right| = 1$ for the generators $C = a X + b Y$. 

 We fix $W_+ = (1 + \kappa c) X - (\kappa + c) Y$ and $W_- = (1 - \kappa c ) X + (\kappa - c) Y$.
Taking into account that $(X, X) = (Y, Y) = 1$ and $(X, Y) = c$, it is easy to check that
$(W_+, \kappa X + Y) = 0$ and $(W_-, \kappa X - Y) = 0$. 

% When $\kappa = 0$, we will add an additional technical requirement that the optimal decomposition must not only minimize the total cost, 
% but also the sum of rotation angles around $Y$-axis (among all decompositions with the minimum cost). Otherwise an optimal decomposition may be 
% multiplied by a trivial factor $R(2 \pi Y)$, which has zero cost. We would like to eliminate this redundancy.

Now we can state the main results of the paper. It turns out that the problem we consider has a bifurcation at $\kappa = c$, and we need to consider the cases $0 \leq c < \kappa \leq 1$ and
$0 < \kappa \leq c < 1$ separately. There will be also a special case when $\kappa = 0$.

We will give the solution of the above optimal control problem by specifying the patterns of optimal decomposition (\ref{dec}).

We begin with some elementary observations. Obviously we may restrict all angles of 
rotation to be less or equal to $\pi$.

 If $g = R(t_1 C_1) R(t_2 C_2) \ldots R(t_n C_n)$ is an optimal decomposition then a decomposition
\begin{equation}
\label{sub}
R (t^\prime_k C_k) R(t_{k+1} C_{k+1}) \ldots R(t_{m-1} C_{m-1}) R(t^\prime_m C_m)
\end{equation}
with $1 \leq k \leq m \leq n$, $0 \leq t^\prime_k \leq t_k$, $0 \leq t^\prime_m \leq t_m$, is also optimal. We call (\ref{sub}) a {\it subword} of $R(t_1 C_1) R(t_2 C_2) \ldots R(t_n C_n)$.
We shall present the optimal decompositions as subwords of certain patterns. 

Since the number of patterns can be fairly large, we shall use various symmetries in order to group several patterns together.
For example, if we have an optimal decomposition
$$g = R(t_1 C_1) R(t_2 C_2) \ldots R(t_n C_n)$$
with $C_j \in \CC$, then
$$R(- t_1 C_1) R(- t_2 C_2) \ldots R(- t_n C_n)$$
is also an optimal decomposition (for a different element of $SO(3)$). This follows from the fact that
multiplication of controls by $-1$ is an automorphism of our problem. We denote this symmetry transformation on the set of patterns by $(X, Y) \mapsto (-X, -Y)$. 

Whereas the set of optimal patterns is always invariant with respect to the symmetry $(X, Y) \mapsto (-X, -Y)$,  
other types of symmetries that we shall consider are not universal and are present only for some patterns. 
If we make the following schematic representation of the controls, all symmetries that we consider
will be elements of the dihedral group of symmetries of a square:
\begin{figure}[h]
% \captionsetup{labelformat=empty}
\includegraphics[clip=true,trim=30pt 30pt 30pt 30pt]{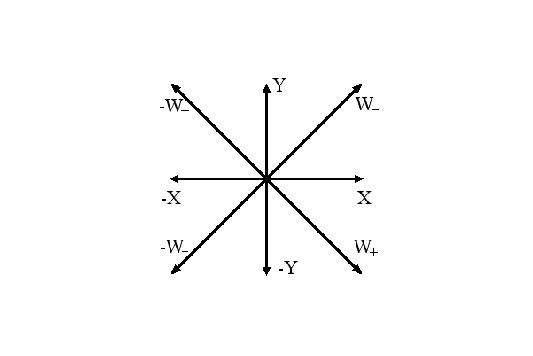} 
\caption{}
\end{figure}

Consider a transformation $X \mapsto X$, $Y \mapsto -Y$, $W_+ \mapsto W_-$,
$W_- \mapsto W_+$. 
Together with the symmetry $(X, Y) \mapsto (-X, -Y)$ this generates a set of $4$ transformations. We denote this set of symmetries by $(X, Y) \mapsto (\pm X, \pm Y)$. We assume that all symmetries we consider are compatible with multiplication by $-1$, even though they are not linear in general.

 We also consider a transformation $X \mapsto Y$, $Y \mapsto X$, $W_+ \mapsto - W_+$, $W_- \mapsto W_-$. Together with $(X, Y) \mapsto (-X, -Y)$, this generates a set of $4$ transformations, which we denote by $(X, Y) \mapsto \left\{ -X, -Y \right\}$.

 Finally, if we consider all of the above transformations together, we generate a full set of $8$ symmetries of the square in Fig.1, which we denote by $(X, Y) \mapsto \left\{ \pm X, \pm Y \right\}$. 

\begin{theorem}
\label{czero}
 Let $c = 0$, $0 < \kappa \leq 1$. For an element $g \in SO(3)$  
the infimum of the optimization Problem 1 is attained on a subword of one of the following patterns:

(I) $R(t_X X) R(t_Y Y) R(- t_X X) R(- t_Y Y)$ where $\tan ( t_X / 2) = \kappa \tan (t_Y /2)$, $0 < t_X, t_Y \leq \pi$, and symmetric to it under $(X, Y) \mapsto \left\{ \pm X, \pm Y \right\}$.

(II)  $R (\pi X) R(t W_+) R(\pi X)$ , with $t \geq 0$, and symmetric to it under $(X, Y) \mapsto \left\{ \pm X, \pm Y \right\}$.

(III)  $R (\pi X) R(t W_+) R(-\pi Y)$ , with $t \geq 0$, and symmetric to it under $(X, Y) \mapsto \left\{ \pm X, \pm Y \right\}$.
\end{theorem}

\

 When we apply symmetry transformations, e.g. $X \mapsto Y$, $Y \mapsto X$, we change the parameters $t_X$, $t_Y$ accordingly, but the relation 
$\tan ( t_X / 2) = \kappa \tan (t_Y /2)$ in (I) is preserved. Under this symmetry transformation, the pattern (I) takes form
%\break
$R(t_Y Y) R(t_X X) R(- t_Y Y) R(- t_X X)$ with $\tan ( t_X / 2) = \kappa \tan (t_Y /2)$.

Set
\begin{equation}
\label{tim}
\htx = \arccos \left(  \frac{c - \kappa} {c + \kappa} \right), \quad
\hty = \arccos \left( - \frac{1 - \kappa c}{1 + \kappa c} \right), \quad
0 \leq \htx, \hty \leq \pi.
\end{equation}

\begin{theorem}
\label{Mgt}
 Let $0 < c < \kappa \leq 1$. 
For an element $g \in SO(3)$ the infimum of 
the optimization Problem 1 is attained on a subword of either pattern (I) or one of the following:

%I. $R(t_X X) R(t_Y Y) R(- t_X X) R(- t_Y Y)$ where $\tan ( t_X / 2) = \kappa \tan (t_Y /2)$, and symmetric to it under $(X, Y) \mapsto \left\{ \pm X, \pm Y \right\}$.

(IV)  $R (\hty Y) R( \htx X) R( t W_+ ) R(\htx X) R(\hty Y)$, with $t \geq 0$, and symmetric to it under $(X, Y) \mapsto \left\{ -X, -Y \right\}$.

(V) $R (\htx Y) R( \htx X) R( t W_+ ) R(-\hty Y) R(-\htx X)$, with $t \geq 0$, and symmetric to it under $(X, Y) \mapsto \left\{ -X, -Y \right\}$.

(VI) $R (\pi Y) R(t W_-) R(\pi Y)$ , with $t \geq 0$, and symmetric to it under $(X, Y) \mapsto \left\{ -X, -Y \right\}$.

(VII)  $R (\pi X) R(t W_-) R(\pi Y)$ , with $t \geq 0$, and symmetric to it under $(X, Y) \mapsto \left\{ -X, -Y \right\}$.
\end{theorem}

\

%The same transformation applied to (II) yields the pattern
%$R (\htx X) R( \hty Y) R(- t W_+ ) R(\hty Y) R(\htx X)$, with $t \geq 0$.

\begin{theorem}
\label{Mlt}
 Let $0 < \kappa \leq c < 1$.
For an element $g \in SO(3)$ the infimum of 
the optimization Problem 1 is attained on a subword of either
patterns (I), (IV), (V) given above, or the following pattern

(VIII)  $R (\pi Y) R(t X) R(\pi Y)$ , with $0 \leq t \leq 2 \htx$, and symmetric to it under $(X, Y) \mapsto \left( -X, -Y \right)$.
\end{theorem}

\

\begin{theorem}
\label{Meq}
 Let $\kappa = 0$, $c \geq 0$. For an element $g \in SO(3)$  
the infimum of the optimization Problem 1 is attained on a subword of one of the following two patterns

(IX)  $R (\pi Y) R(t W_+) R(\pi Y)$ , with $t \geq 0$, and symmetric to it under $(X, Y) \mapsto ( \pm X, \pm Y )$.

(X)  $R (\pi Y) R(t W_+) R(-\pi Y)$ , with $t \geq 0$, and symmetric to it under $(X, Y) \mapsto ( \pm X, \pm Y )$.
\end{theorem}

\begin{remark}
When $\kappa = c > 0$ we have $W_-$ proportional to $X$, and the lists of patterns in Theorems \ref{Mgt} and \ref{Mlt} 
become equivalent.
\end{remark}

% \begin{remark}
% An optimal decomposition of $g \in SO(3)$ needs not to be unique. For example, using Proposition \ref{Wconj} below, 
% we get the equalities:
% \begin{align*}
% & R(t_1 X) R(\hty Y) R(-t_2 W_+) R(-t_3 X) \\
% & = R((t_1 - \htx) X) R(t_2 W_+) R(\htx X) R (\hty Y) R(-t_3 X) \\
% & =  R((t_1 - \htx) X) R(t_2 W_+) R(- \hty Y) R(- (t_3 + \htx) X) \\
% & = R((t_1 - \htx) X) R( t_4 W_+) R( \htx X) R(\hty Y) R( - (t_2 - t_4) W_+) R(-t_3 X).
% \end{align*}
% If $t_1 > \htx$, $t_2 > t_4 > 0$ and $t_3 > 0$, all of these decompositions have the same cost and could be optimal.
% \end{remark}

\section{Geometric optimization theory}

In this section we will review the geometric optimization theory following \cite{J}, and apply it to our optimization problem.

 The Lie algebra $so(3)$ of the Lie group $SO(3)$ is the tangent space to $SO(3)$ at identity and consists of skew-symmetric $3\times 3$ matrices. The Lie bracket of two matrices in $so(3)$ is
$[A,B] = AB - BA$. We may identify the space $so(3)$ with $\R^3$ via the map (\ref{cross}) $X \mapsto \ad (X)$. The corresponding Lie bracket of two vectors in $\R^3$ is the cross product.

% Given a smooth curve $U(t) \in SO(3)$ the derivative $U^\prime (t)$ belongs to the tangent space
% to $SO(3)$ at $U(t)$, and thus we can write $U^\prime (t) = U(t) u(t)$ for some $u(t) \in so(3)$.

 Fix $g \in SO(3)$. A curve leading to $g$ is an absolutely continuous function $U : [0, t_0] \rightarrow SO(3)$ such that $U(0) = I$ and $U(t_0) = g$. An absolutely continuous function has a measurable derivative $u: \ [0, t_0] \rightarrow so(3)$ such that $U^\prime (t) = U(t) u(t)$ for almost all $t$. The derivative $u$
is Lebesgue integrable \cite{S}.

 Let us formulate a differential version of our optimization problem.

{\bf Problem 2.} For an element $g \in SO(3)$ find the infimum of $\int_0^{t_0} \cost(U^{-1}(t) U^\prime(t)) dt$
\break
 over all absolutely continuous curves  
$U : [0, t_0] \rightarrow SO(3)$ leading to $g$, satisfying 
\break
$U^{-1}(t) U^\prime(t) \in \CC \subset \R^3 = so(3)$ for almost all $t$.

The parameter $t_0$ in Problem 2 is not fixed and when taking the infimum we consider the curves with all $t_0 \geq 0$.

 It is clear that the restriction to the case of piecewise constant controls $u(t) = U^{-1}(t) U^\prime(t)$ gives precisely Problem 1. On the other hand we shall see that the solutions of Problem 2 indeed have piecewise constant controls, which implies equivalence of Problems 1 and 2.  

\begin{proposition}
\label{existence}
For any $g \in SO(3)$ there exists an absolutely continuous optimal solution $U$ for Problem 2.
\end{proposition}
\begin{proof}
The proof of this Proposition is based on the observation that the cost assigned to a curve $U: [0, t_0] \rightarrow SO(3)$ is independent of the choice of its
parametrization. To prove this, we first note the cost function (\ref{costf}) satisfies $\cost (\lambda u) = \lambda \cost (u)$ for $\lambda \geq 0$. Consider
an absolutely continuous increasing surjective reparametrization $f: [0, \tau_0] \rightarrow [0, t_0]$ and the corresponding reparametrized curve $\bU (\tau) = U (f (\tau))$. 
Then $\bU$ and $U$ have the same cost:
\begin{align*}
& \int_{0}^{\tau_0} \cost \left( \bU (\tau)^{-1} \frac{d}{d\tau} \bU (\tau) \right) d \tau \\
& =  \int_{0}^{\tau_0} \cost \left( U (t)^{-1} \frac{d}{dt} U (t) \big|_{t = f(\tau)} f^\prime (\tau) \right) d \tau \\
& =  \int_{0}^{\tau_0} \cost \left( U (t)^{-1} \frac{d}{dt} U (t) \big|_{t = f(\tau)} \right)  f^\prime (\tau) d \tau \\
& = \int_0^{t_0} \cost \left( U (t)^{-1} U^\prime (t)  \right)  d t.
\end{align*}
This computation shows that rescaling of the set of controls $\CC$ does not change the cost of a curve $U$ leading to $g$ with $U^{-1}(t) U^\prime (t) \in \CC$.

Let us modify Problem 2 by replacing the set $\CC$ with its convex hull
\begin{equation*}
\overline{\CC} = \left\{ a X + b Y \; \big| \; |a| + |b| \leq 1 \right\} . 
\end{equation*}
Once the control set is convex, we can apply Theorem 4.10 from \cite{S} to obtain the existence of an absolutely continuous optimal
solution $U: [0, t_0] \rightarrow SO(3)$ for the modified problem. To go back to the setting of Problem 2, we note that every absolutely
continuous curve admits a parametrization by the arc length, i.e., the natural parametrization (see for example Section 5.3 in \cite{P}). 
Then it is easy to see that the curve $U$ may also be reparametrized with $U^{-1}(t) U^\prime (t) \in \CC$. Since reparametrization does not
change the cost, we see that an optimal solution of the modified problem with the set of controls $\overline{\CC}$ yields an optimal solution for
Problem 2.
\end{proof}

\begin{remark}
Our optimization problem induces a left-invariant metric on $SO(3)$. It is possible to see that this metric
does not correspond to any Riemannian structure on this Lie group. 
\end{remark}

 The Hamiltonian function for Problem 2 is 
\begin{equation*}
\Ham (p, u) = p_0 \cost (u) + (p, u),
\quad u \in \CC, p \in \R^3 ,
\end{equation*}
which involves a parameter $p_0 \leq 0$ (see Section 11.2.2 in \cite{J} for details).

 For each $p \in \R^3$ we define the maximal Hamiltonian
\begin{equation*}
\M (p) = \max_{u \in \CC} \Ham (p, u).
\end{equation*}

\begin{theorem}
\label{PMP}
(Pontryagin's Maximum Principle, \cite{J})
Let $U$ be an optimal curve leading to $g \in SO(3)$ for Problem 2. Then 
there exists an absolutely continuous function $p: [0, t_0] \rightarrow \R^3 = so(3)$ and a constant $p_0 \leq 0$ such that  for almost all $t \in [0, t_0]$
the following equations hold:

\begin{equation}\LeftEqNo
\tag{i} 
\Ham (p(t), u(t)) = \M (p(t)) = 0
\end{equation}
and 
\begin{equation}\LeftEqNo
\tag{ii} 
\frac{dp}{dt} = p(t) \times u(t).
\end{equation}

If $p_0 = 0$ then $p(t)$ is non-zero for almost all $t \in [0, t_0]$.

\end{theorem}

\begin{lemma}
\label{cons}
The quantity $|p(t)|$ is conserved.
\end{lemma}
\begin{proof}
\begin{equation*}
\frac{d}{dt} |p(t)|^2 = 2 \left( \frac{dp}{dt} , p(t) \right) = 
2 \left( p(t) \times u(t) , p(t) \right) = 0 .
\end{equation*}
\end{proof}

Note that for our problem the parameter $p_0$ can not be zero, otherwise condition (i) implies that
$\max_{u \in \CC} (p(t), u) = 0$, hence $p(t)$ is proportional to $Z$ for almost all $t$, and so is 
$\frac{dp}{dt}$. However (ii) implies that $\left( \frac{dp}{dt}, p(t) \right) = 0$ and thus
$\frac{dp}{dt} = 0$ and $p(t)$ is a constant multiple of $Z$. Inspecting (ii) again, we conclude that
 $p(t)$ must be zero for almost all $t$, which contradicts the last claim of the theorem.

 In case when the parameter $p_0$ is non-zero, it can be rescaled to any negative value. A convenient choice for us is $p_0 = -\sin^2 (\alpha)$.

 Consider a second basis $\{ S, Q \}$ of the $XY$-plane, where
\begin{equation*}
S = Y \times Z = X - cY \quad \hbox{\rm and} \quad Q = - X \times Z = Y - cX.
\end{equation*}
 Then $(S, X) = (Q, Y) = (S, S) = (Q, Q) = \sin^2 (\alpha)$,  $(S, Y) = (Q, X) = 0$. In this basis $W_+ = (1 + \kappa c) X - (\kappa + c) Y =S - \kappa Q$ and $W_- =  (1 - \kappa c) X + (\kappa - c) Y =S + \kappa Q$.

 According to Theorem \ref{PMP}, the value of $p(t)$ determines the value of $u(t)$ via (i), while by (ii) the value of $u(t)$ determines the evolution of $p(t)$. Let us analyze (i) to see which values of
$p(t)$ are admissible, and what are the corresponding controls $u$.

 Let us write $u = a X + b Y$ and $p = s S + q Q + z Z$. Then
\begin{equation*}
\Ham(p, u) = \sin^2 (\alpha) \left( - \left| a \right| - \kappa   \left| b \right| + sa +qb \right) .
\end{equation*} 

Since the set $\CC$ is closed under symmetry $a \mapsto -a$, $b \mapsto -b$, we see that the maximum in $u \in \CC$ of $H (p, u)$ is attained when $a$ has the same sign as $s$ and $b$ has the same sign as $q$. Hence
\begin{equation*}
\M (p) / \sin^2 (\alpha) = \max_{|a| + |b| = 1} (|s| -1) |a| + (|q| - \kappa) |b| = \max \left\{ |s| - 1, |q| - \kappa \right\} .
\end{equation*}

 By property (i) of the Theorem, $\M (p(t)) = 0$, thus the admissible values of $p(t)$ satisfy
either $|s| = 1$, $|q| \leq \kappa$ or $|q| = \kappa$, $|s| \leq 1$. 
We summarize this in the following Lemma, which  describes controls in the resulting regions:
\begin{lemma}
\label{reg}
(a) Let $\kappa > 0$. 

(i) If $s = 1$, $-\kappa < q < \kappa$ then $a=1$, $b = 0$, the control is $u = X$;

(ii) If $s = -1$, $-\kappa < q < \kappa$ then $a=-1$, $b = 0$, the control is $u = -X$;

(iii) If $q = \kappa$, $-1 < s < 1$ then $a=0$, $b = 1$, the control is $u = Y$;

(iv) If $q = -\kappa$, $-1 < s < 1$ then $a=0$, $b = -1$, the control is $u = -Y$.

\

(b)  If $\kappa = 0$ then $q = 0$, $|s| \leq 1$. When $q=0$ and $-1 < s < 1$
we could have either control $u = Y$ or $u = -Y$. 
\end{lemma}

 At the points where two regions meet, the whole segment joining the corresponding two controls is allowed. For example, when $s = 1$ and $q = \kappa > 0$ we could have any control $u = a X + b Y$
with $a, b \geq 0$, $a+b=1$. We will call such values of $p$ {\it critical}.

 If the curve $p(t)$ reaches a critical point, one of three things could happen: the curve $p(t)$ could
cross the boundary of a region, in which case the control will switch; the curve $p(t)$ could return to the same region where it came from without a switch of control; or the curve $p(t)$ may stay inside the critical boundary for some positive time. Let us describe evolution of $p(t)$ inside the critical boundary.

\begin{lemma}
\label{critreg}
(a) Suppose $p(t) = S - \kappa Q + z(t) Z$ for $t \in [t_1, t_2]$. Then $u(t) = W_+$ and $z(t) = 0$
for  $t \in [t_1, t_2]$. 

(b) Suppose $p(t) = S + \kappa Q + z(t) Z$ for $t \in [t_1, t_2]$ and $\kappa > 0$. Then $\kappa \geq c$,  $u(t) = W_-$ and $z(t) = 0$ for  $t \in [t_1, t_2]$.
\end{lemma}

Cases $s(t) = -1$, $q(t) = \kappa$ and $s(t) = -1$, $q(t) = -\kappa$ are analogous, the controls are $u(t) = - W_+$ and $u(t) = - W_-$ respectively and $z(t) = 0$.

% Note that when $\kappa = 0$ we have $W_+ = W_- = X - cY$, and the two parts of the Lemma coincide. 

\begin{proof}
To prove (a) consider equation (ii) in Theorem \ref{PMP}. We get
\begin{equation*}
\frac{d}{dt} \left( S - \kappa Q + z(t) Z \right) = \left(  S - \kappa Q + z(t) Z \right) \times
\left( a X + b Y \right) .
\end{equation*}

Taking into account that
\begin{equation*}
S \times X = cZ, \quad S \times Y = Z, \quad Q \times X = -Z, \quad Q \times Y = - cZ,
\end{equation*}
we get that
\begin{equation*}
\frac{dp}{dt} = a z(t) Q - b z(t) S + (ca + \kappa a + b + c\kappa b) Z .
\end{equation*}
Since $q(t)$ and $s(t)$ are constant, this implies $z(t) = 0$ for  $t \in [t_1, t_2]$. Then we get
$(c+\kappa) a = - (1+ c\kappa) b$ and $u$ is proportional to $W_+$.

Case (b) is analogous, except that for $\kappa < c$ the segment joining $X$ and $Y$ does not contain a vector proportional to $W_-$.

\end{proof}

%\captionsetup[subfigure]{labelformat=empty}
\renewcommand{\thesubfigure}{\Alph{subfigure}}

\begin{figure}
\centering
\begin{subfigure}{3in}
\includegraphics[clip=true,trim=40pt 0pt 30pt 0pt]{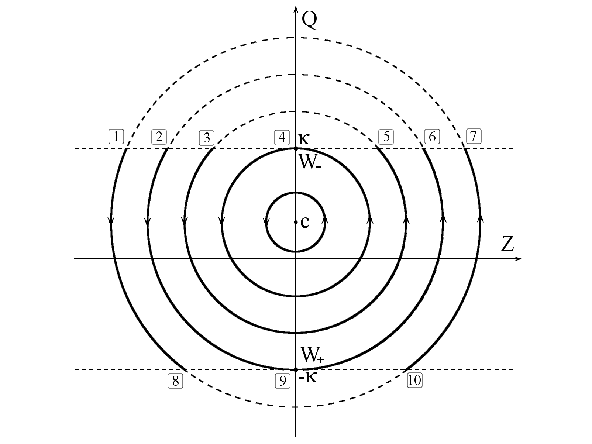} 
\caption{Evolution with $u = X$, $s = 1$}
\label{EX}
\end{subfigure}
\begin{subfigure}{3in}
\includegraphics[clip=true,trim=40pt 0pt 30pt 0pt]{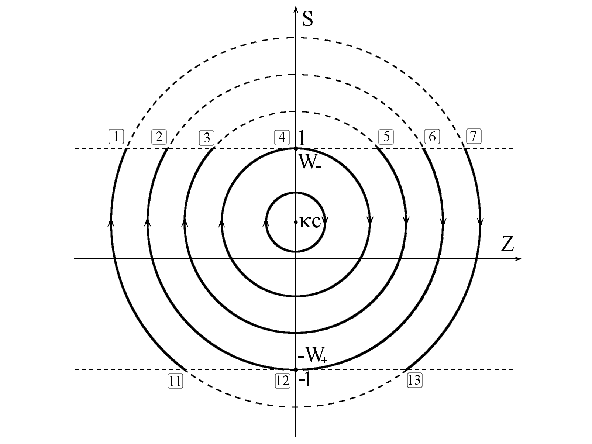} 
\caption{Evolution with $u = Y$, $q = \kappa$}
\end{subfigure} \\
\centering
\begin{subfigure}{3in}
\includegraphics[clip=true,trim=40pt 0pt 30pt 0pt]{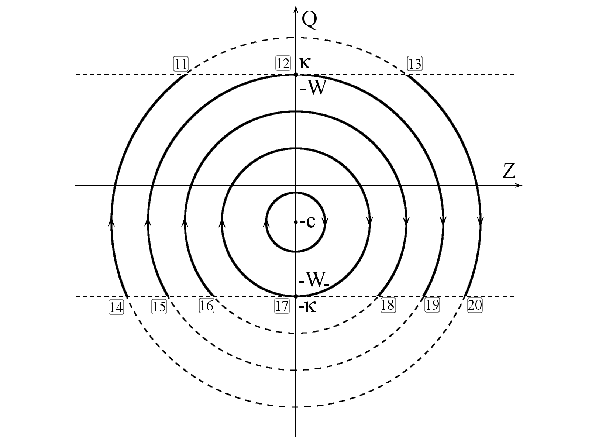} 
\caption{Evolution with $u = -X$, $s = -1$}
\end{subfigure}
\begin{subfigure}{3in}
\includegraphics[clip=true,trim=40pt 0pt 30pt 0pt]{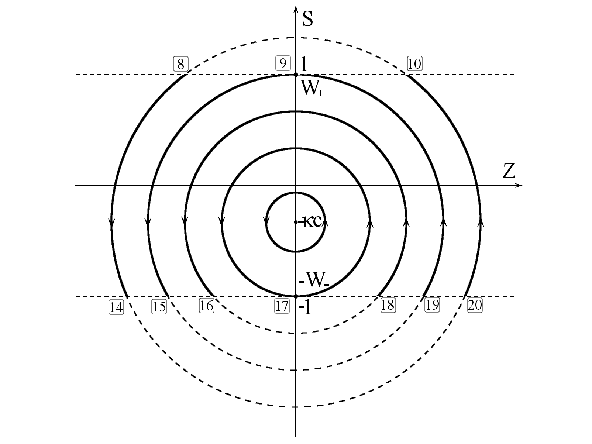} 
\caption{Evolution with $u = -Y$, $q = -\kappa$}
\end{subfigure}
\caption{Case $\kappa > c$.}
\label{kgtc}
\end{figure}

\begin{corollary}
An optimal solution of Problem 2 could only involve controls $\pm X$, $\pm Y$, $\pm W_+$ and $\pm W_-$. Moreover, controls $\pm W_-$
do not occur if $0< \kappa < c$.
\end{corollary}

Note that when $\kappa = c$ we get $W_-$ proportional to $X$. When $\kappa = 0$ we get $W_+ = W_-$.

Next, let us study evolution of $p(t)$ under controls $\pm X$ and $\pm Y$.

As we have seen in Lemma \ref{reg}, control $X$ corresponds to the region $s = 1$, $-\kappa \leq q \leq \kappa$. Let $p(t) = S + q(t) Q + z(t) Z$.
By part (ii) of Theorem \ref{PMP}, evolution of $p(t)$ is given by
\begin{equation*}
\frac{dp}{dt} = \left(  S + q(t) Q + z(t) Z \right) \times X = (c - q(t)) Z + z(t) Q.
\end{equation*} 
From this we get
\begin{equation*}
q^\prime (t) = z(t), \quad z^\prime (t) = - (q(t) - c), \quad s^\prime (t) = 0.
\end{equation*}
Setting $\tilde{q} (t) = q(t) - c$, we get the equations of the harmonic oscillator
\begin{equation*}
\tilde{q}^\prime (t) = z(t), \quad z^\prime (t) = - \tilde{q} (t) 
\end{equation*}
with solutions $q(t) = c + K \sin(t + \theta)$, $z(t) = K \cos(t+\theta)$. We plot the trajectories in $QZ$-plane in Fig.\ref{EX} and \ref{FX}. Similarly, we plot the trajectories for the other regions described in Lemma \ref{reg}.

%Old figure
% %\captionsetup[subfigure]{labelformat=empty}
% \renewcommand{\thesubfigure}{\Alph{subfigure}}
% 
% \setcounter{figure}{0}
% \begin{figure}[labelformat=empty]
% \ContinuedFloat
% \begin{subfigure}{3in}
% \includegraphics[clip=true,trim=40pt 0pt 30pt 0pt]{Evolution_X_k_gt_c.eps} 
% \caption{Evolution with $u = X$, $s = 1$}
% \label{EX}
% \end{subfigure}
% \begin{subfigure}{3in}
% \includegraphics[clip=true,trim=40pt 0pt 30pt 0pt]{Evolution_Y_k_gt_c.eps} 
% \caption{Evolution with $u = Y$, $q = \kappa$}
% \end{subfigure}
% \end{figure}
% \setcounter{figure}{0}
% \begin{figure}
% \ContinuedFloat
% \begin{subfigure}{3in}
% \setcounter{subfigure}{2}
% \includegraphics[clip=true,trim=40pt 0pt 30pt 0pt]{Evolution_-X_k_gt_c.eps} 
% \caption{Evolution with $u = -X$, $s = -1$}
% \end{subfigure}
% \begin{subfigure}{3in}
% \includegraphics[clip=true,trim=40pt 0pt 30pt 0pt]{Evolution_-Y_k_gt_c.eps} 
% \caption{Evolution with $u = -Y$, $q = -\kappa$}
% \end{subfigure}
% \caption{Case $\kappa > c$}
% \label{kgtc}
% \end{figure}

 This gives us the trajectories that satisfy the conditions of Theorem \ref{PMP}. For example, the path 
\hbox{
\mybox{\kern -18pt} $1$ 
{\kern 0pt} $\mapsto$ {\kern -8pt}
\mybox{\kern -18pt} $8$
{\kern 0pt} $\mapsto$ {\kern -8pt}
\mybox{\kern -21pt} $14$
{\kern -1pt} $\mapsto$ {\kern -8pt}
\mybox{\kern -21pt} $11$
{\kern -1pt} $\mapsto$ {\kern -8pt}
\mybox{\kern -18pt} $1$ 
{\kern 2pt}}
corresponds to the decomposition 
$$R(t_1 X) R(-t_2 Y) R(-t_3 X) R(t_4 Y) .$$
When a trajectory reaches a critical point, for example 
\hbox{\mybox{\kern -18pt} $4$ {\kern 2pt}}, 
it could continue from 
\hbox{\mybox{\kern -18pt} $4$ {\kern 2pt}} 
either using evolution with controls $X$, $Y$ or remain at this critical point for some positive time using control $W_-$. 

The conservation law of Lemma \ref{cons} ensures that for the trajectory 
\hbox{
\mybox{\kern -18pt} $1$ 
{\kern 0pt} $\mapsto$ {\kern -8pt}
\mybox{\kern -18pt} $8$
{\kern 0pt} $\mapsto$ {\kern -8pt}
\hbox{\mybox{\kern -21pt} $14$}
{\kern 2pt}}
the points 
\hbox{\mybox{\kern -18pt} $1$  {\kern 2pt}} and  
\hbox{\mybox{\kern -21pt} $14$  {\kern 2pt}}
have equal $Z$-coordinates. The same property holds in other similar cases, and in particular the trajectory that starts at a critical point  
\hbox{\mybox{\kern -18pt} $9$  {\kern 2pt}} and goes to 
\hbox{\mybox{\kern -21pt} $15$  {\kern 2pt}}
will reach the critical point  
\hbox{\mybox{\kern -21pt} $12$  {\kern 2pt}}. 

It follows that for the trajectory 
\hbox{
\mybox{\kern -18pt} $1$ 
{\kern 0pt} $\mapsto$ {\kern -8pt}
\mybox{\kern -18pt} $8$
{\kern 0pt} $\mapsto$ {\kern -8pt}
\mybox{\kern -21pt} $14$
{\kern -1pt} $\mapsto$ {\kern -8pt}
\mybox{\kern -21pt} $11$
{\kern 2pt}}
evolution times for the parts 
\hbox{
\mybox{\kern -18pt} $1$ 
{\kern 0pt} $\mapsto$ {\kern -8pt}
\mybox{\kern -18pt} $8$ 
{\kern 2pt}}
and 
\hbox{
\mybox{\kern -21pt} $14$
{\kern -1pt} $\mapsto$ {\kern -8pt}
\mybox{\kern -21pt} $11$
{\kern 2pt}}
are the same, since the corresponding arcs are symmetric to each other.

 Next we establish the relations between the time parameters in these trajectories (cf. Proposition 2.1 in \cite{B}).

\begin{proposition}
\label{times}
(a) Let $t_X$ be the $X$-evolution time, and $t_Y$ be $-Y$-evolution time for the trajectory
\hbox to 85pt{\mybox{\kern -18pt} $1$ 
{\kern 0pt} $\mapsto$ {\kern -8pt}
\mybox{\kern -18pt} $8$
{\kern 0pt} $\mapsto$ {\kern -8pt}
\mybox{\kern -22pt} $14$
{\kern 2pt}}.
Then 
\begin{equation*}
\tan (t_X /2) = \kappa \tan (t_Y/2) .
\end{equation*}
The same relation holds for the trajectories
 \hbox to 88pt{\mybox{\kern -21pt} $14$ 
{\kern -1pt} $\mapsto$ {\kern -8pt}
\mybox{\kern -21pt} $11$
{\kern -1pt} $\mapsto$ {\kern -8pt}
\mybox{\kern -18pt} $1$
{\kern 2pt}},
 \hbox to 88pt{\mybox{\kern -18pt} $7$ 
{\kern 0pt} $\mapsto$ {\kern -8pt}
\mybox{\kern -21pt} $13$
{\kern -1pt} $\mapsto$ {\kern -8pt}
\mybox{\kern -21pt} $20$
{\kern 3pt}},
\break
 \hbox to 89pt{\mybox{\kern -21pt} $13$ 
{\kern -1pt} $\mapsto$ {\kern -8pt}
\mybox{\kern -21pt} $20$
{\kern -1pt} $\mapsto$ {\kern -8pt}
\mybox{\kern -21pt} $10$
{\kern 2pt}}, etc., with $t_X$ being the time parameter for $\pm X$-evolution and 
$t_Y$ for $\pm Y$-evolution.

(b) Let $\htx$ be the time of evolution for the trajectories involving critical points,
\hbox to 60pt{
\mybox{\kern -18pt} $9$ 
{\kern 0pt} $\mapsto$ {\kern -8pt}
\mybox{\kern -18pt} $6$
{\kern 2pt}},
\hbox to 60pt{
\mybox{\kern -18pt} $2$ 
{\kern 0pt} $\mapsto$ {\kern -8pt}
\mybox{\kern -18pt} $9$
{\kern 2pt}},
\hbox to 60pt{
\mybox{\kern -21pt} $15$ 
{\kern 0pt} $\mapsto$ {\kern -8pt}
\mybox{\kern-21pt} $12$
{\kern 2pt}} 
or
\hbox to 60pt{
\mybox{\kern -21pt} $12$ 
{\kern -1pt} $\mapsto$ {\kern -8pt}
\mybox{\kern -21pt} $19$
{\kern 2pt}}.

 Let $\hty$ be the time of evolution for the trajectories 
\hbox to 60 pt{
\mybox{\kern -21pt} $12$ 
{\kern -1pt} $\mapsto$ {\kern -8pt}
\mybox{\kern -18pt} $2$
{\kern 2pt}},
\hbox to 60pt{
\mybox{\kern -18pt} $6$ 
{\kern 0pt} $\mapsto$ {\kern -8pt}
\mybox{\kern -21pt} $12$
{\kern 2pt}},
\hbox to 60pt{
\mybox{\kern -21pt} $19$ 
{\kern -1pt} $\mapsto$ {\kern -8pt}
\mybox{\kern -18pt} $9$
{\kern 2pt}}
 or
\hbox to 60pt{
\mybox{\kern -18pt} $9$ 
{\kern 0pt} $\mapsto$ {\kern -8pt}
\mybox{\kern -21pt} $15$
{\kern 2pt}}.

Then 
\begin{equation}
\label{timm}
\htx = \arccos \left(  \frac{c - \kappa} {c + \kappa} \right), \quad
\hty = \arccos \left( - \frac{1 - \kappa c}{1 + \kappa c} \right).
\end{equation}

\end{proposition}

\begin{proof}
Consider the trajectory 
\hbox{\mybox{\kern -18pt} $1$ 
{\kern 0pt} $\mapsto$ {\kern -8pt}
\mybox{\kern -18pt} $8$
{\kern 0pt} $\mapsto$ {\kern -8pt}
\mybox{\kern -21pt} $14$
{\kern 2pt}}.
Let $z_1$ and $z_2$ be $Z$-coordinates of the points \hbox{\mybox{\kern -18pt} $1$ {\kern 0pt}},
and 
\hbox{\mybox{\kern -18pt} $8$ {\kern 0pt}} respectively.
Then $z_1$ is also the  $Z$-coordinate of the point
\hbox{\mybox{\kern -21pt} $14$ {\kern 2pt}}.

\begin{figure}
\centering
\begin{subfigure}{3in}
\includegraphics[clip=true,trim=40pt 0pt 30pt 0pt]{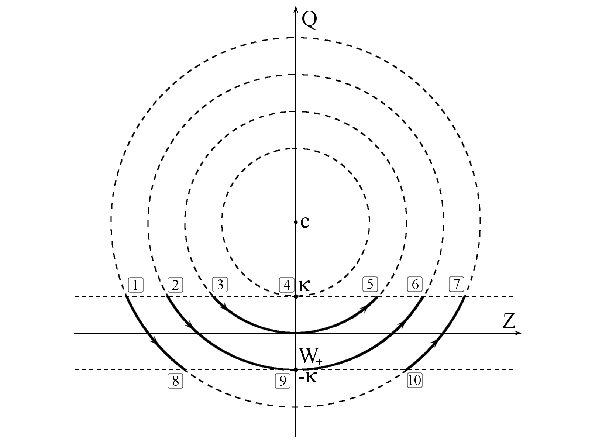} 
\caption{Evolution with $u = X$, $s = 1$}
\label{FX}
\end{subfigure}
\begin{subfigure}{3in}
\includegraphics[clip=true,trim=40pt 0pt 30pt 0pt]{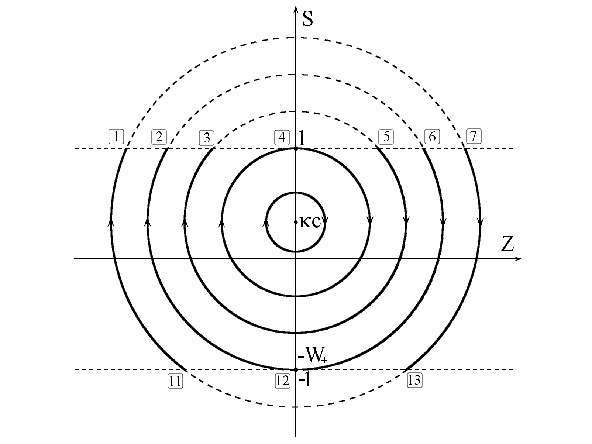} 
\caption{Evolution with $u = Y$, $q = \kappa$}
\end{subfigure} \\
\centering
\begin{subfigure}{3in}
\includegraphics[clip=true,trim=40pt 0pt 30pt 0pt]{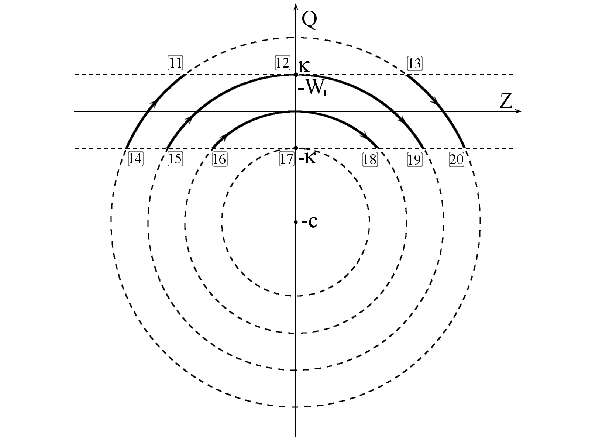} 
\caption{Evolution with $u = -X$, $s = -1$}
\end{subfigure}
\begin{subfigure}{3in}
\includegraphics[clip=true,trim=40pt 0pt 30pt 0pt]{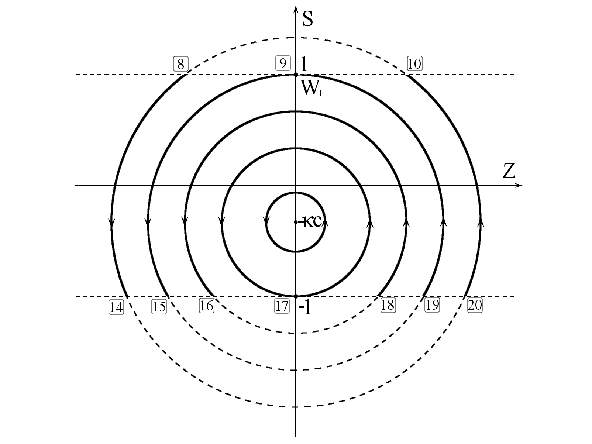} 
\caption{Evolution with $u = -Y$, $q = -\kappa$}
\end{subfigure}
\caption{Case $0< \kappa \leq c$.}
\label{kltc}
\end{figure}

 Since the points \mybox{\kern -19pt} $1$ {\kern 2pt} and
\mybox{\kern -19pt} $8$ {\kern 2pt} lie on a circle with the center at $Z = 0$, $Q = c$,
they satisfy the equation
\begin{equation}
\label{cir}
z_1^2 + (\kappa - c)^2 = z_2^2 + (\kappa + c)^2 .
\end{equation} 
%Likewise, the points 
%\mybox{\kern -19pt} $8$ {\kern 0pt}
%and
%\mybox{\kern -22pt} $14$ {\kern 2pt}
%satisfy the equation
%\begin{equation}
%z_2^2 + (1+\kappa c)^2 = z_1^2 + (1-\kappa c)^2 .
%\end{equation}
Let $b$ be the base of the isosceles triangle with vertices at 
\mybox{\kern -19pt} $1$ {\kern 2pt},
\mybox{\kern -19pt} $8$ {\kern 2pt}
and the center of the circle, and let $h$ be the altitude in this triangle.
Then
\begin{equation*}
b^2 = (z_2 - z_1)^2 + (2 \kappa)^2
\end{equation*}
and 
\begin{equation*}
h^2 = \left( \frac{z_1+z_2}{2} \right)^2 + c^2 .
\end{equation*}
Since $t_X$ is the angle at the vertex of this triangle, we have
\begin{equation*}
\tan^2\left( \frac{t_X} {2}\right) = \left( \frac{b}{2h} \right)^2 = 
\frac { (z_2 - z_1)^2 + 4 \kappa^2} { (z_2 + z_1)^2 + 4 c^2}.
\end{equation*}
Similarly,
\begin{equation*}
\tan^2\left( \frac{t_Y} {2}\right) = 
\frac { (z_2 - z_1)^2 + 4} {(z_2 + z_1)^2 + 4 \kappa^2 c^2} .
\end{equation*}
Since $\tan (t_X /2), \tan (t_Y /2) > 0$, in order to establish claim (a), we need
to show that  $\tan^2 (t_X /2) = \kappa^2 \tan^2 (t_Y /2)$. 
This equality however follows from (\ref{cir}):
\begin{multline*}
\left(  (z_2 - z_1)^2 + 4 \kappa^2 \right)
\left( (z_2 + z_1)^2 + 4 \kappa^2 c^2 \right) - 
\kappa^2 \left(  (z_2 - z_1)^2 + 4 \right)
\left(  (z_2 + z_1)^2 + 4 c^2 \right) \\
= (\kappa^2 - 1) \left( (z_2 - z_1)^2 (z_2 + z_1)^2 - 16 \kappa^2 c^2 \right) = 0.
\end{multline*} 
The proof for the other cases in (a) is analogous. 

\begin{figure}
%\begin{figure}[labelformat=empty]
\begin{subfigure}{3in}
\includegraphics[clip=true,trim=40pt 0pt 30pt 0pt]{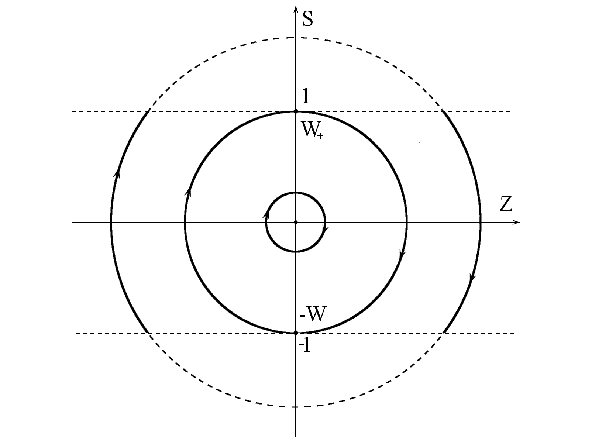} 
\caption{Evolution with $u = Y$, $q = 0$}
%\label{EX}
\end{subfigure}
\begin{subfigure}{3in}
\includegraphics[clip=true,trim=40pt 0pt 30pt 0pt]{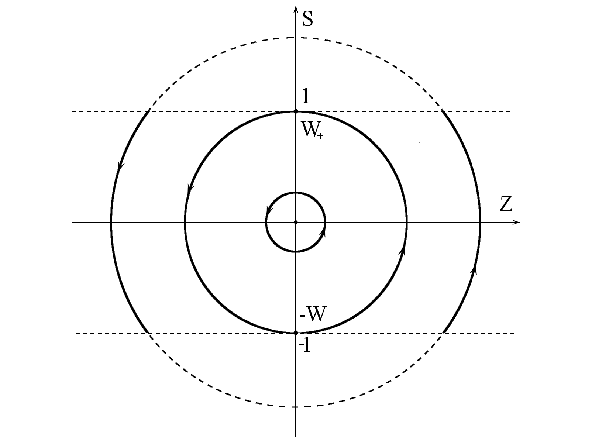} 
\caption{Evolution with $u = -Y$, $q = 0$}
\end{subfigure}
\caption{Case $ \kappa = 0$.}
\label{keqz}
\end{figure}

Let us prove part (b). Consider the trajectory 
\hbox{
\mybox{\kern -18pt} $9$ 
{\kern 0pt} $\mapsto$ {\kern -8pt}
\mybox{\kern -18pt} $6$
{\kern 2pt}}. 
The time parameter $\htx$ is the angle corresponding to this arc of the circle  with center at $Z = 0$, $Q = c$ and radius $\kappa + c$. Taking the projection to $Q$-axis we get
\begin{equation*}
\cos(\htx) =  \frac{c - \kappa} {c + \kappa} .
\end{equation*}
The derivation of the formula for $\hty$ is analogous.
\end{proof}

 The patterns listed in Theorem \ref{Mgt} can be traced on the diagrams in Fig.\ref{kgtc}, while the patterns of Theorem \ref{Mlt} can be seen on Fig. \ref{kltc}.
For example, the pattern given in part (I) of Theorem \ref{Mgt} corresponds to the trajectory
\hbox{\mybox{\kern -21pt} $10$ 
{\kern -1pt} $\mapsto$ {\kern -8pt}
\mybox{\kern -18pt} $7$
{\kern 0pt} $\mapsto$ {\kern -8pt}
\mybox{\kern -21pt} $13$
{\kern -1pt}  $\mapsto$ {\kern -8pt}
\mybox{\kern -21pt} $20$
{\kern -1pt} $\mapsto$ {\kern -8pt}
\mybox{\kern -21pt} $10$
{\kern 2pt}}. 
The above Proposition describes the relations between the time parameters of the evolution. 

To complete the proofs of Theorems \ref{czero} -- \ref{Meq} we need to show that decompositions with a large number of switches can not be optimal. We defer this to the next section.

% Now let us discuss the case $0 < \kappa \leq c < 1$. Here the evolution diagrams 
% given by Theorem \ref{PMP} are given in Fig. \ref{kltc}
% and Proposition \ref{times} remains valid.

%  Consider the case when $\kappa = 0$. It follows from Lemmas \ref{reg} (b) and \ref{critreg} (a) that possible controls are $\pm Y$ and $\pm W_+$.
% (note that control $Y$ has zero cost when $\kappa = 0$). 
% Note that  $W_+ = W_- = X - cY$ is a vector orthogonal to $Y$. The evolution diagrams in this case are given by Fig. \ref{keqz}.
% The claim of Theorem \ref{Meq} immediately folows from these diagrams.

\section{Bounds on the number of control switches}

In this section we are going to show that certain decompositions are not optimal, even though
they satisfy the necessary conditions of the Pontryagin's Maximum Principle. 
This will give us constraints on the number of control switches in optimal decompositions. 

Rather than doing computations in the group of rotations $SO(3)$, it is easier to carry them  out in the unitary group $SU(2)$, which is a double cover of $SO(3)$:
\begin{equation}
\label{SUSO}
\varphi: \ SU(2) \rightarrow SO(3) .
\end{equation} 

 Let us recall the construction of $SU(2)$ based on the quaternions.
 The algebra of quaternions $\HH$ has a basis $\{ 1, i, j, k \}$ and 
relations $i^2 = j^2 = k^2 = -1$, $ij =-ji = k$, $jk = -kj = i$, $ki = - ik = j$.
Similar to the complex numbers, we have the conjugation on $\HH$, given by
$\overline{1} = 1$, $\overline{i} = -i$, $\overline{j} = -j$, $\overline{k} = -k$, 
and the norm: $| ai + bj + ck + d| = \sqrt{a^2 + b^2 + c^2 + d^2}$. Every non-zero element of $\H$ has a multiplicative inverse given by
$w^{-1} = \overline{w}/{|w|^2}$.

The unitary group $SU(2)$ may be realized as a unit sphere in the quaternion algebra
$\HH$:
\begin{equation*}
SU(2) = \left\{ ai + bj + ck + d \; | \; a^2 + b^2 + c^2 + d^2 = 1 \right\}.
\end{equation*}

The Lie algebra $su(2)$ of the group $SU(2)$ is the tangent space at identity, it is a 3-dimensional subspace in $\HH$ spanned by $\{ i, j, k \}$. 
We are going to identify this Lie algebra with $\R^3$ via $i \mapsto e_1$, $j \mapsto e_2$, $k \mapsto e_3$, where 
$\{ e_1, e_2, e_3 \}$ is the standard basis of $\R^3$.
Since
$[i, j ] = ij - ji = 2k$, $[j,k] = jk - kj = 2i$, $[k, i] = ki - ik = 2j$, we see that two Lie
algebra structures on $\R^3$ coming from $so(3)$ and $su(2)$ differ by a factor of $2$. For this reason
there is a factor of $2$ in the formula for the homomorphism $\varphi$:
\begin{equation*}
\varphi \left( \exp (X) \right) = R (2 X).
\end{equation*}
Here for a vector $X = (a, b, c)^T$ the exponential is computed in the algebra of quaternions 
$\exp (X) = \exp ( ai + bj + ck ) \in SU(2)$.
Note that the rotation operator $R(X)$ is also an exponential:
$R(X) = \exp (\ad (X))$.

The kernel of the homomorphism $\varphi$ is $\{ \pm 1\} \subset SU(2)$, so the map $\varphi$ is 2 to 1.

The advantage of using $SU(2)$ instead of $SO(3)$ is that $SU(2)$ is embedded in a 4-dimensional vector space $\HH$, while $SO(3)$ is embedded into the 9-dimensional space of $3 \times 3$ matrices.

 In our computations we are going to use the Campbell-Hausdorff formula \cite{K} (up to the second order terms):
\begin{equation}
\label{CH}
\exp( \ep A ) \exp (\ep B) = \exp ( C), \quad \hbox{\rm where \ } C = \ep (A+B) 
+ \frac{\ep^2}{2} [A, B] + \oo (\ep^2) .
\end{equation}

 We will also need the conjugation formula:
$C \exp(B) C^{-1} = \exp \left( C B C^{-1} \right)$.

 Pontryagin's Maximum Principle that we use above is essentially a local first derivative test. In order to obtain stronger results, we need to either apply non-local transformations (those that do not come from a small variation of parameters) or use higher derivatives.  In Proposition \ref{epsfive} we will be using the second derivative in order to show that certain decompositions are not optimal.
An example of a non-local transformation is the identity $R(\pi X) = R(-\pi X)$ where
$\left| X \right| = 1$. This trivial observation may be generalized in the following way.
Suppose $R(t_1 X) R(t_2 Y)$ is a rotation in angle $\pi$. Then we get a relation
$R(t_1 X) R(t_2 Y) = R(-t_2 Y) R(-t_1 X)$. Note that both sides of this equality have the same cost. This non-local relation and its consequences will be quite useful for our analysis.

\begin{lemma}
\label{dz}
Let $g = ai + bj + ck + d \in SU(2)$. The image of $g$ in $SO(3)$ is a rotation in angle $\pi$ if and only if $d = 0$.
\begin{proof}
Clearly, $\varphi (g)$ is a rotation in angle $\pi$ if and only if $\varphi (g)^2$ is the
identity matrix, but $\varphi (g)$ is not identity. This is equivalent to $g^2 = \pm 1$,
$g \neq \pm 1$ in $SU(2)$. It is easy to see that the only solutions to $g^2 = 1$ are
$g = \pm 1$. Thus the preimages of rotations in angle $\pi$ are precisely $g \in SU(2)$ with $g^2 = -1$, or equivalently, $g^{-1} = -g$. Since $\left| g \right| = 1$,
this becomes $\overline{g} = -g$. For $g = ai + bj + ck + d$ this is equivalent to $d = 0$.
\end{proof}
\end{lemma}

\begin{proposition}
\label{flip}
Let $X, Y \in \R^3 = su(2)$. Suppose $|X| = |Y| = 1$ and let $\alpha$ be the angle between $X$ and $Y$. 

(a) If $\tan (s_1) \tan (s_2) = \frac{1}{\cos(\alpha)}$ then 
the image of $\exp (s_1 X) \exp (s_2 Y)$ in $SO(3)$ is a rotation in angle $\pi$.
In this case  $\exp (s_1 X) \exp (s_2 Y) = - \exp (-s_2 Y) \exp (-s_1 X)$.

(b) Let $\tan \psi = \cos \alpha \tan (s_2)$, $-\frac{\pi}{2} < \psi \leq \frac{\pi}{2}$. Then
\begin{equation*}
\exp( s_1 X) \exp (s_2 Y) \exp (s_3 X) = 
\exp(s_1^\prime X) \exp (-s_2 Y) \exp (s_3^\prime X),
\end{equation*}
where $s_1^\prime = s_1 + \psi - \frac{\pi}{2}$,  $s_3^\prime = s_3 + \psi + \frac{\pi}{2}$.
% \begin{equation}
% \cos (\psi) = \frac{\cos (s_2)}{\sqrt{\cos^2 (s_2) + \cos^2 \alpha \sin^2 (s_2)}},
% \sin (\psi) = \frac{\cos \alpha \sin (s_2)}{\sqrt{\cos^2 (s_2) + \cos^2 \alpha \sin^2 (s_2)}}.
% \end{equation}
\end{proposition}
\begin{proof}
The group $SU(2)$ acts on its Lie algebra $su(2)$ by conjugation, and its center $\{ \pm 1 \}$ acts trivially. This gives the action of $SO(3)$ on $su(2)$, which is the natural action of $SO(3)$ on $\R^3$. Since this action is transitive on pairs of unit vectors with a given angle between them, we may set without loss of generality $X = i$, $Y = i \cos \alpha + j \sin \alpha  $. We complete this to a basis of
$su(2)$ by setting $Z = \frac{1}{2} [X, Y] = k \sin \alpha$. We can easily verify that
\begin{equation}
\label{rel}
\begin{split}
XY = -\cos \alpha + Z, \quad & \quad YX = -\cos \alpha - Z, \\
XYX = Y - 2 \cos \alpha X, \quad & \quad YXY = X - 2 \cos \alpha Y, \\
[Z, X] =  2 Y - 2 \cos \alpha X, \quad & \quad [Z, Y] =  - 2 X + 2 \cos \alpha Y, \\
XZX = Z, \quad & \quad YZY = Z .
\end{split}
\end{equation} 
We also note that $X^2 = -1$ and $\exp (sX) = \cos(s) + X \sin(s)$ and likewise for $Y$.

We have
\begin{multline*}
\exp (s_1 X) \exp (s_2 Y)  = 
\left( \cos(s_1) + X \sin(s_1) \right) \left( \cos(s_2) + Y \sin(s_2) \right) \\
= \left( \cos(s_1) \cos(s_2) - \cos(\alpha) \sin (s_1) \sin(s_2) \right) \\
+ X \sin(s_1) \cos(s_2) + Y \cos(s_1) \sin(s_2) + Z \sin(s_1) \sin(s_2) .
\end{multline*}
Applying Lemma \ref{dz} we establish the claim of part (a).

 Using part (a), we get
\begin{equation*}
\exp (s_2 Y) = - \exp (- \tau X) \exp (-s_2 Y) \exp( - \tau X),
\end{equation*}
where $\tan (\tau) \tan (s_2) = \frac{1} {\cos (\alpha)}$. Set $\psi = \frac{\pi}{2} - \tau$. Then $\tan \psi = \frac{1}{\tan \tau} = \cos \alpha \tan (s_2)$ and
\begin{multline*}
\exp (s_2 Y) =
 - \exp ( (\psi - \frac{\pi}{2}) X) \exp (-s_2 Y) \exp( (\psi - \frac{\pi}{2}) X) \\
= \exp ( (\psi - \frac{\pi}{2}) X) \exp (-s_2 Y) \exp( (\psi + \frac{\pi}{2}) X). 
\end{multline*} 
Multiplying both sides by $\exp (s_1 X)$ on the left and $\exp (s_3 X)$ on the right,
we get the claim of part (b).
\end{proof}

\begin{proposition}
\label{epsthree}
Let $\tan \left| \frac{t_X}{2} \right| = \kappa \tan \left| \frac{t_Y}{2} \right|$ .
Decompositions
$R(t_Y Y) R(t_X X) R(-t_Y Y)$ with $|t_Y| > \frac{\pi}{2}$ 
and  $R(t_X X) R(t_Y Y) R(-t_X X)$ with $|t_X| > \frac{\pi}{2}$ 
are not optimal.
\end{proposition}
\begin{proof}
We may assume without loss of generality that $t_X, t_Y > 0$.
Let us begin with the case of $R(t_Y Y) R(t_X X) R(-t_Y Y)$. 
We take its preimage under $\varphi$:
$\exp (s_1 Y) \exp (s_2 X) \exp (-s_1 Y) \in SU(2)$, where
$s_1 = t_Y/2$, $s_2 = t_X/2$, $\tan(s_2) = \kappa \tan (s_1)$, $s_1 > \frac{\pi}{4}$.
 We claim that the decomposition
$\exp (s_1^\prime Y) \exp (-s_2 X) \exp (s_3^\prime Y)$ given by the previous proposition will have a lower cost. Since $|s_2| = |-s_2|$ we need to show that
$|s_1^\prime| + |s_3^\prime| < 2|s_1|$, where
\begin{equation}
\label{sprime}
s_1^\prime = s_1 + \psi - \frac{\pi}{2}, \quad s_3^\prime = -s_1 + \psi + \frac{\pi}{2}. 
\end{equation}
% We shall assume that $s_1, s_2 >0$, all other cases are completely similar.
We have $\psi > 0$ and $s_3^\prime > 0$. If $s_1^\prime < 0$ then
\begin{equation*}
|s_1^\prime| + |s_3^\prime| = - (s_1 + \psi - \frac{\pi}{2}) + (-s_1 + \psi + \frac{\pi}{2})
= \pi - 2 s_1 < \frac{\pi}{2} < 2 s_1,
\end{equation*}
and we get that the new cost is lower. 
If $s_1^\prime \geq 0$ then
\begin{equation}
\label{newcost}
|s_1^\prime| + |s_3^\prime| = (s_1 + \psi - \frac{\pi}{2}) + (-s_1 + \psi + \frac{\pi}{2})
= 2\psi .
\end{equation}
Since $\tan \psi = \cos \alpha \tan s_2$ and $\tan s_2 = \kappa \tan s_1$, we get that
$\psi < s_2 \leq s_1$, so the new cost is again lower. 

 We now apply the same approach to 
$R(t_X X) R(t_Y Y) R(-t_X X)$. We again take its preimage
$\exp (s_1 X) \exp (s_2 Y) \exp (-s_1 X)$ in $SU(2)$ and transform it into
 $\exp (s_1^\prime X) \exp (-s_2 Y) \exp (s_3^\prime X)$ using Proposition \ref{flip}. Here $\tan s_1 = \kappa \tan s_2$. The values of $s_1^\prime$, $s_3^\prime$ are still given by (\ref{sprime}) with $\tan \psi = \cos \alpha \tan s_2$. We have $s_3^\prime > 0$ and consider the sign of $s_1^\prime$. The case $s_1^\prime < 0$ is treated in the same way as before. 

When $s_1^\prime \geq 0$ we consider two subcases: $\kappa \leq \cos \alpha$ and $\kappa > \cos \alpha$. If $\kappa \leq \cos \alpha$, the claim of the Proposition follows from the observation that on the diagrams (A), (C) in Fig. \ref{kltc} the arcs
\hbox{
\mybox{\kern -18pt} $1$ 
{\kern 0pt} $\mapsto$ {\kern -8pt}
\mybox{\kern -18pt} $8$ 
{\kern 2pt}},
\hbox{
\mybox{\kern -21pt} $10$
{\kern -1pt} $\mapsto$ {\kern -8pt}
\mybox{\kern -18pt} $7$
{\kern 2pt}},
\hbox{
\mybox{\kern -21pt} $14$
{\kern -1pt} $\mapsto$ {\kern -8pt}
\mybox{\kern -21pt} $11$
{\kern 2pt}}
and
\hbox{
\mybox{\kern -21pt} $13$
{\kern -1pt} $\mapsto$ {\kern -8pt}
\mybox{\kern -21pt} $20$
{\kern 2pt}}
correspond to an angle not exceeding $\frac{\pi}{2}$. 

Let us assume $\kappa > \cos \alpha$. To show that the transformed expression has a lower cost, we need to 
prove that $\psi < s_1$. However $\tan \psi = \cos \alpha \tan s_2 =
\frac{\cos \alpha}{\kappa} \tan s_1$. Since $\frac{\cos \alpha}{\kappa} < 1$, we
get $\psi < s_1$, which completes the proof of the Proposition.
\end{proof}

\begin{proposition}
\label{epsfive}
Let $\delta > 0$ be a small parameter and let $\kappa \neq 0$. Then the decompositions
\begin{equation}
\label{fivedec}
R ( -\delta X) R( t_Y Y) R(t_X X) R(-t_Y Y) R( -\delta X)
\end{equation}
with $t_X, t_Y > 0$, $\tan (t_X/2) = \kappa \tan (t_Y/2)$, and those symmetric to it under $(X, Y) \mapsto \{ \pm X, \pm Y \}$, are not optimal.
\end{proposition}
\begin{proof}
Let us assume by contradiction that the given decomposition is optimal. 
As before, we take a preimage $\exp (-\ep X) \exp (s_1 Y) \exp(s_2 X) \exp( - s_1 Y) \exp (- \ep X)$, where  $\ep = \delta/2$, $s_1 = t_Y /2$, $s_2 = t_X /2$. 
% A necessary condition for  (\ref{fivedec}) to be optimal is $\tan s_2 = \kappa \tan s_1$. We also must have $\kappa \neq 0$.
We shall express the given decomposition in the following way:
\begin{multline}
\label{twoexp}
\exp (-\ep X) \exp (s_1 Y) \exp(s_2 X) \exp( - s_1 Y) \exp (- \ep X) \\
= \exp ((s_1 + \ep_1) Y) \exp((s_2 + \ep_2) X) \exp( - (s_1+\ep_2^\prime) Y) \exp (- \ep_3 X) .
\end{multline}
We are going to solve for $\ep_1$, $\ep_2$, $\ep_2^\prime$ and $\ep_3$ in terms of $\ep$, and show that the new decomposition has a lower cost. Note that the parameters $\ep_2$ and $\ep_2^\prime$ are bound by the relation
$\tan (s_2 + \ep_2) = \kappa \tan (s_1 + \ep_2^\prime)$. 

We will use the Campbell-Hausdorff formula (\ref{CH}) to rewrite both sides of (\ref{twoexp}) in the form
\begin{equation*}
\exp (s_1 Y) \exp(s_2 X) \exp (L) \exp( - s_1 Y).
\end{equation*}
We shall calculate $L$ up to the second order in $\ep$. Applying (\ref{CH}) to the left hand side of (\ref{twoexp}), we get that
\begin{equation}
\label{firstL}
L = L_1 + L_2 + \frac{1}{2} [L_1, L_2] + \oo (\ep^2),
\end{equation}
where
\begin{equation*}
L_1 = \exp (-s_2 X) \exp(-s_1 Y) (- \ep X )  \exp (s_1 Y) \exp(s_2 X),
\end{equation*}
and 
\begin{equation*}
L_2 = \exp(-s_1 Y)  (- \ep X )  \exp (s_1 Y) .
\end{equation*}
Let us carry out the detailed calculations. We shall use the basis $\{ X, Y, Z \}$ 
and relations (\ref{rel}) as in the proof of the Proposition \ref{flip}.
\begin{multline*}
L_2 = - \ep \left( \cos (s_1) - Y \sin (s_1) \right) X  \left( \cos (s_1) + Y \sin (s_1) \right) \\
= -\ep \left( X \cos (2s_1) + Z \sin (2s_1) + 2 c Y \sin^2 (s_1) \right) .
\end{multline*}
Next,
\begin{multline*}
L_1 =  \exp (-s_2 X) L_2 \exp(s_2 X) \\
= - \ep \left( \cos (s_2) - X \sin(s_2) \right) 
\left( X \cos (2s_1) + Z \sin (2s_1) + 2 c Y \sin^2 (s_1) \right)
 \left( \cos (s_2) + X \sin(s_2) \right) \\
= - \ep \big(
X \left( \cos (2s_1) - c \sin (2s_1) \sin (2s_2) + 4 c^2 \sin^2 (s_1) \sin^2 (s_2) \right) \\
+ Y \left( \sin (2s_1) \sin( 2s_2) + 2 c \sin^2 (s_1) \cos (2s_2) \right) 
+ Z \left( \sin (2s_1) \cos (2s_2) - 2c \sin^2 (s_1) \sin (2 s_2) \right) \big) . 
\end{multline*}
Doing the same calculations for the right hand side of (\ref{twoexp}), we get
\begin{equation}
\label{secondL}
L = L_3 + L_4 + L_5 + L_6 + \frac{1}{2} \left( [L_3, L_4] + [L_3 , L_5]
+ [L_3, L_6] + [L_4, L_5] + [L_4, L_6] + [L_5, L_6] \right)  + \oo (\ep^2),
\end{equation}
where
\begin{align*}
& L_3 = \exp (- s_2 X) (\ep_1 Y) \exp (s_2 X), \\
& L_4 = \ep_2 X, \quad L_5 = - \ep_2^\prime Y, \\
& L_6 = \exp ( - s_1 Y ) ( - \ep_3 X) \exp (s_1 Y) .
\end{align*}
We begin by solving (\ref{twoexp}) to the first order in $\ep$. Equating (\ref{firstL}) with (\ref{secondL}) we get:
\begin{multline*}
-\ep X \left(
2 \cos(2 s_1) - c \sin (2 s_1) \sin (2 s_2) + 4 c^2 \sin^2 (s_1) \sin^2 (s_2) \right) \\
- \ep Y \left(
\sin (2 s_1) \sin (2 s_2) + 4 c \sin^2 (s_1) \cos^2 (s_2) \right) 
- \ep Z \left( 2 \sin (2 s_1) \cos^2 (s_2) - 2 c \sin^2 (s_1) \sin (2 s_2) \right) \\
= X \left( 2 \ep_1 c \sin^2 (s_2) + \ep_2 - \ep_3 \cos( 2 s_1) \right) 
+ Y \left( \ep_1 \cos (2 s_2) - \ep_2^\prime - 2 \ep_3 c \sin^2 (s_1) \right) \\
+ Z \left( - \ep_1 \sin (2 s_2) - \ep_3 \sin (2 s_1) \right) .
\end{multline*}

We divide both sides of this equation by $\cos^2 (s_1) \cos^2 (s_2)$, which
allows us to express everything in terms of $\tan (s_1), \tan (s_2)$. Using the relation
$\tan (s_2) = \kappa \tan (s_1)$, we further eliminate $\tan (s_2)$. To make the equations more compact we denote $\tan (s_1)$ by $\px$.
By Proposition \ref{epsthree} we have $0 < \px \leq 1$.

Since we also have the relation $\kappa \tan (s_1 + \ep_2^\prime) = \tan (s_2 + \ep_2)$, we use the Taylor expansion to find the relation between $\ep_2$ and $\ep_2^\prime$ to the first order:
\begin{equation*}
\frac{ \kappa \ep_2^\prime}{\cos^2 (s_1)} =
\frac{\ep_2}{\cos^2 (s_2)} + \oo(\ep_2).
\end{equation*}
Expressing this in terms of $\px$, we get
\begin{equation*}
\ep_2^\prime = \ep_2 \frac {\kappa \px^2 + \kappa^{-1}} {\px^2 + 1} + \oo(\ep_2).
\end{equation*}
Equating the coefficients at $X, Y, Z$, we get a system of equations
\begin{multline*}
2 \ep_1 c (1+\px^2) \kappa^2 \px^2 + \ep_2 (1+\px^2)(1+ \kappa^2 \px^2)
- \ep_3 (1 + \kappa^2 \px^2) (1 - \px^2) 
\hfill
\\
\hfill
= - \ep \left( 2 (1 - \px^2) (1 + \kappa^2 \px^2) - 4 c \kappa \px^2 + 4 c^2 \kappa^2 \px^4 \right) + \oo (\ep), \\
\ep_1 (1+\px^2)(1 - \kappa^2 \px^2) - \ep_2 \kappa^{-1} (1 + \kappa^2 \px^2 )^2
- 2 \ep_3 c \px^2 (1 + \kappa^2 \px^2) 
= - 4 \ep ( \kappa + c) \px^2 + \oo (\ep), \\ 
\hfill
- 2 \ep_1 \kappa \px (1+\px^2) - 2 \ep_3 \px (1 + \kappa^2 \px^2) 
= - \ep \left( 4 \px - 4c \kappa \px^3 \right) + \oo (\ep) .
\hfill
\end{multline*}
The determinant of this system is the jacobian of (\ref{twoexp}) and equals
\begin{equation*}
4 \px (1 + \px^2) (1 + \kappa^2 \px^2)^2 (1 - \kappa^2 \px^4 + 2 c \kappa \px^2
+ c \kappa \px^4 + c \kappa^3 \px^4).
\end{equation*}
Since $0 < \px \leq 1$, $0 < \kappa \leq 1$, $0 \leq c < 1$, we see that the only case when the jacobian vanishes is $\px = 1$, $\kappa =1$, $c = 0$. We will consider this case separately below. In all other cases the jacobian is non-zero, hence by the Implicit Function Theorem, equation (\ref{twoexp}) has a unique solution for small $\ep$.

Solving (\ref{twoexp}) to the second order in $\ep$, we get that the cost of the right hand side of  (\ref{twoexp}) is
\begin{multline*}
\kappa (s_1 + \ep_1) + (s_2 + \ep_2) + \kappa (s_1+\ep_2^\prime) + \ep_3 \\
= 2 \kappa s_1 + s_2 + 2 \ep
- 2 \ep^2 \frac
{\kappa \px (1 - c \kappa \px^2) (2 - \kappa^2 \px^4 + \kappa^2 \px^2 
+ 3 c^2 \px^2 + c^2 \px^4 + 2 c^2 \kappa^2 \px^4)}
{(1+\px^2) (1 - \kappa^2 \px^4 + 2c \kappa \px^2 + c \kappa \px^4 + c \kappa^3 \px^4)} + \oo(\ep^2),
\end{multline*}
which is lower than the cost of the left hand side
$2 \kappa s_1 + s_2 + 2 \ep$.

 For the remaining case $\px = 1$, $\kappa =1$, $c = 0$, we have 
$s_1 = s_2 = \frac{\pi}{4}$.

Applying Proposition \ref{flip} we see that
\begin{equation*}
\exp \left( \frac{\pi}{4} Y \right) 
\exp \left( \frac{\pi}{4} X \right) 
\exp \left( - \frac{\pi}{4} Y \right) =
\exp \left( - \frac{\pi}{4} Y \right) 
\exp \left( - \frac{\pi}{4} X \right) 
\exp \left( \frac{\pi}{4} Y \right), 
\end{equation*}
and thus
\begin{multline*}
\exp \left( - \ep X \right)
\exp \left( \frac{\pi}{4} Y \right) 
\exp \left( \frac{\pi}{4} X \right) 
\exp \left( - \frac{\pi}{4} Y \right)
\exp \left( - \ep X \right) \\
 =
\exp \left( - \ep X \right)
\exp \left( - \frac{\pi}{4} Y \right) 
\exp \left( - \frac{\pi}{4} X \right) 
\exp \left( \frac{\pi}{4} Y \right)
\exp \left( - \ep X \right) , 
\end{multline*}
which is not optimal since it does not correspond to any trajectory in Fig. \ref{kgtc}.

 This completes the proof of the proposition for the decomposition
\break
$R ( -\delta X) R( t_Y Y) R(t_X X) R(-t_Y Y) R( -\delta X)$. The cases of
the decompositions obtained from this one by applying symmetries 
$(X, Y) \mapsto \{ \pm X, \pm Y \}$ are analogous. For example, in the case of $R ( -\delta Y) R( t_X X) R(t_Y Y) R(-t_X X) R( -\delta Y)$,
we use the transformation
\begin{multline}
\label{twoexptwo}
\exp (-\ep Y) \exp (s_1 X) \exp(s_2 Y) \exp( - s_1 X) \exp (- \ep Y) \\
= \exp ((s_1 + \ep_1) X) \exp((s_2 + \ep_2) Y) \exp( - (s_1+\ep_2^\prime) X) \exp (- \ep_3 Y) .
\end{multline}
The cost of the right hand side is
\begin{equation*}
2 s_1 + \kappa s_2 + 2 \kappa \ep
- 2 \ep^2 \frac
{\kappa \px ( 1 - c \kappa \px^2) ( 2 + \px^2 - \kappa^2 \px^4 
+ 3 c^2 \kappa^2 \px^2 + 2 c^2 \kappa^2 \px^4 + c^2 \kappa^4 \px^4)}
{(1 + \kappa^2 \px^2) ( 1 - \kappa^2 \px^4 + c \kappa \px^4 + 2 c \kappa \px^2 + c \kappa^3 \px^4)}
+ \oo (\ep^2),
\end{equation*}
which is lower than the cost $2 s_1 + \kappa s_2 + 2 \kappa \ep$
of the left hand side.
\end{proof}

It follows from Proposition \ref{epsfive} that for the optimal decompositions corresponding to the trajectory 
\hbox{
$\ldots \mapsto$
\mybox{\kern -18pt} $1$ 
{\kern 0pt} $\mapsto$ {\kern -8pt}
\mybox{\kern -18pt} $8$
{\kern 0pt} $\mapsto$ {\kern -8pt}
\mybox{\kern -21pt} $14$
{\kern -1pt} $\mapsto$ {\kern -8pt}
\mybox{\kern -21pt} $11$
{\kern -1pt}  $\mapsto$ {\kern -8pt}
\mybox{\kern -18pt} $1$ 
{\kern 0pt} $\mapsto$ {\kern -8pt}
$\ldots$},
and symmetric to it,
the number of factors is at most $4$. This corresponds to pattern (I) in Theorems
\ref{czero} -- \ref{Mlt}.

Suppose $\kappa > c$. For the trajectory 
\hbox{
$\ldots \mapsto$
\mybox{\kern -18pt} $3$ 
{\kern 0pt} $\mapsto$ {\kern -8pt}
\mybox{\kern -18pt} $5$
{\kern 0pt} $\mapsto$ {\kern -8pt}
\mybox{\kern -18pt} $3$
{\kern 0pt} $\mapsto$ {\kern -8pt}
$\ldots$},
and symmetric to it,
the number of factors is bounded by $2$, since the evolution times 
\hbox{
\mybox{\kern -18pt} $3$ 
{\kern 0pt} $\mapsto$ {\kern -8pt}
\mybox{\kern -18pt} $5$
{\kern 2pt}}
and
\hbox{
\mybox{\kern -18pt} $5$ 
{\kern 0pt} $\mapsto$ {\kern -8pt}
\mybox{\kern -18pt} $3$
{\kern 2pt}}
exceed $\pi$, and optimal decompositions can not have such time parameters. The case of two factors is incorporated in patterns
(III) and (VII) with $t=0$ in Theorems \ref{czero} and \ref{Mgt}.

In the case $0 < \kappa \leq c$, the decompositions corresponding to the trajectory
\break
\hbox{
\mybox{\kern -18pt} $5$ 
{\kern 0pt} $\mapsto$ {\kern -8pt}
\mybox{\kern -18pt} $3$
{\kern 0pt} $\mapsto$ {\kern -8pt}
\mybox{\kern -18pt} $5$
{\kern 2pt}},
could have up to $3$ factors, since $Y$-evolution time
\hbox{
\mybox{\kern -18pt} $5$ 
{\kern 0pt} $\mapsto$ {\kern -8pt}
\mybox{\kern -18pt} $3$
{\kern 2pt}}
exceeds $\pi$, but $X$-evolution time does not exceed $2 \htx$, which is less than $\pi$. This corresponds to pattern (VIII) in Theorem \ref{Mlt}.

%  To finish the proof of Theorem \ref{czero} we consider evolution through the critical points. Here we note that the decomposition (II) $R(\pi X) R(tW_+) R(\pi X)$
% is itself not optimal (but its subwords are). Indeed, its optimality would imply that $R(-\pi X) R(tW_+) R(-\pi X)$ is also optimal, but the latter decomposition 
% does not satisfy the necessary conditions of the Pontryagin's Maximum Principle, as it does not correspond to any trajectory in Fig. \ref{kgtc}.

To complete the proof of Theorems \ref{czero} -- \ref{Meq}, we need to establish a bound on the number of factors for the trajectories that pass through the critical points. For the critical points $\pm W_-$ with $c \neq 0$ the existence of such a bound  immediately follows from the diagrams in Fig. \ref{kgtc}, since the trajectories connected to these points are full circles and require time evolution of $2 \pi$ to complete the circle, while any evolution with time exceeding $\pi$ is not optimal. 

 In the case of the critical points $\pm W_+$ we need to deal with a 
trajectory
\break
\hbox{
$\ldots \mapsto$
\mybox{\kern -18pt} $9$ 
{\kern 0pt} $\mapsto$ {\kern -8pt}
\mybox{\kern -18pt} $6$
{\kern 0pt} $\mapsto$ {\kern -8pt}
\mybox{\kern -21pt} $12$
{\kern -1pt} $\mapsto$ {\kern -8pt}
\mybox{\kern -21pt} $12$
{\kern -1pt}  $\mapsto$ {\kern -8pt}
\mybox{\kern -21pt} $19$ 
{\kern -1pt} $\mapsto$ {\kern -8pt}
\mybox{\kern -18pt} $9$ 
{\kern 0pt} $\mapsto$ {\kern -8pt}
\mybox{\kern -18pt} $9$ 
{\kern 0pt} $\mapsto$ {\kern -8pt}
\mybox{\kern -21pt} $15$ 
{\kern -1pt} $\mapsto$ {\kern -8pt}
$\ldots$},
and other similar to it. Again, we will establish a bound on the number of switches.

The trajectory
\hbox{  
\mybox{\kern -18pt} $9$ 
{\kern 0pt} $\mapsto$ {\kern -8pt}
\mybox{\kern -18pt} $6$
{\kern 0pt} $\mapsto$ {\kern -8pt}
\mybox{\kern -21pt} $12$
{\kern 2pt}}
corresponds to the product 
\begin{equation*}
R( \htx X) R (\hty Y) \quad \text{with \ }
\htx = \arccos \left(  \frac{c - \kappa} {c + \kappa} \right), \quad
\hty = \arccos \left( - \frac{1 - \kappa c}{1 + \kappa c} \right).
\end{equation*} 
We are going to see that this element of $SO(3)$ is a rotation in angle $\pi$ around an axis, which is orthogonal to $W_+$.

\begin{proposition}
\label{Wconj}
(a) The products $R( \htx X) R (\hty Y)$ and  $R (\hty Y) R( \htx X)$ are both rotations in angle $\pi$.

(b) The following relations hold:
\begin{gather*}
R (\htx X) R (\hty Y) = R ( - \hty Y) R (- \htx X), \\
R (\hty Y) R( \htx X) = R (- \htx X) R (- \hty Y), \\
R (\hty Y) R( \htx X) R ( t W_+) = R (- t W_+) R (\hty Y) R( \htx X), \\
R (\htx X) R (\hty Y) R (- t W_+) = R ( t W_+) R (\htx X) R (\hty Y).
\end{gather*}

(c) Let $\kappa = 0$. Then $R (\pi Y) R  ( t W_+) = R (- t W_+) R (\pi Y)$. 
\end{proposition}

\begin{proof}
Let us consider a preimage $h = \exp (\hsx X) \exp (\hsx Y)$ in $SU(2)$ for 
$R (\htx X) R (\hty Y)$. Here $\hsx = \htx/2$, $\hsy = \hty/2$. It follows from (\ref{tim}) that
\begin{equation*}
% \label{trigs}
\sin \hsx = \sqrt{\frac{\kappa}{\kappa+c}}, \quad
\cos \hsx = \sqrt{\frac{c}{\kappa+c}}, \quad
\sin \hsy = \frac{1}{\sqrt{1 + \kappa c}},  \quad
\cos \hsy =\sqrt{ \frac{\kappa c}{1 + \kappa c}}. 
\end{equation*}
Then
\begin{multline*}
h = \exp (\hsx X) \exp (\hsx Y) = 
\left( \cos \hsx + X \sin \hsx \right)
\left( \cos \hsy + Y \sin \hsy \right) \\
= \frac{1} {\sqrt{(\kappa + c) (1 + \kappa c)}}
\left(  c \sqrt{\kappa} + X \kappa \sqrt{c}
+ Y  \sqrt{c} + XY \sqrt{\kappa} \right) \\
=  \frac{1} {\sqrt{(\kappa + c) (1 + \kappa c)}}
\left(  X \kappa \sqrt{c} + Y  \sqrt{c} + Z \sqrt{\kappa} \right) .
\end{multline*}
By Lemma \ref{dz}, the image of $h$ in $SO(3)$ is a rotation in
angle $\pi$ and the first equality in part (b) holds. 
It is easy to see that $h$ is orthogonal to $W_+$:
\begin{equation*}
\left( \kappa \sqrt{c} X +  \sqrt{c} Y +  \sqrt{\kappa} Z \big| S - \kappa Q \right) = 0,
\end{equation*}
which implies that the axis of rotation corresponding to $h$ is orthogonal to $W_+$
and also that $h W_+ h^{-1} = - W_+$. Taking the exponential of both sides, we get
that $h \exp (s W_+) h^{-1} = \exp (-s W_+)$, from which the third claim of part (b) follows.

The argument for  $R (\hty Y) R( \htx X)$ is completely analogous.

For $\kappa = 0$ we have $(W_+, Y) = 0$, from which the claim (c) follows.
\end{proof}

\begin{proposition}
\label{singleW}
Suppose an element $g \in SO(3)$ has an optimal decomposition containing factors
$R (t W_+)$ or $R (t W_-)$. Then there is an optimal decomposition for $g$
with a single factor of that type.
\end{proposition}
\begin{proof}
First we consider the case $c \neq 0$ and $\kappa \neq 0$. We have pointed out above that the factor 
$R (t W_-)$ may only appear when $c < \kappa$ and there will be only one such a factor in that case. In case when $c = \kappa$, we have that $W_-$ is proportional to $X$, and we do not need to consider the factors of the form $R (t W_-)$ at all.

Consider an optimal decomposition with factors $R(tW_+)$.
Without loss of generality assume that the time parameter in the first such factor is positive. 
Then it will necessarily have the form
\begin{equation*}
h_0 R (t_1 W_+) h_1 R (-t_2 W_+) h_2 R (t_3 W_+) h_3 \ldots 
h_{n-1} R ( (-1)^{n-1} t_n W_+) h_n,
\end{equation*}
where $t_k \geq 0$ and for $1 \leq k \leq n-1$
\begin{equation*}
h_k = \begin{cases}
R (\htx X) R (\hty Y) \mbox{ or } R ( - \hty Y) R (- \htx X) \mbox{ if } k \mbox{ is odd,} \\
R (\hty Y) R (\htx X) \mbox{ or } R (- \htx X) R ( - \hty Y) \mbox{ if } k \mbox{ is even.} 
\end{cases}
\end{equation*}
By Proposition \ref{Wconj} we can combine all factors of type $R (t W_+)$ into one without changing the cost:
\begin{equation*}
h_0 R ((t_1 + t_2 + \ldots + t_n) W_+) h_1  h_2 \ldots h_{n-1}  h_n.
\end{equation*}

Next suppose $\kappa = 0$. Using Proposition \ref{Wconj}(c) and applying the above argument we see that there is an optimal decomposition with at most one factor
of type $R(t W_+)$. This completes the proof of Theorem \ref{Meq}.

Now let us consider the case $c = 0$. Here we could have a decomposition that contains factors of both types, $R( t W_+)$ and $R (t W_-)$. Note that 
$\htx = \hty = \pi$. Suppose that an optimal decomposition of $g \in SO(3)$ contains a factor $R (t W_+)$ with $t > 0$. This factor will be followed by either an $X$-evolution or $-Y$-evolution. Let us assume it is $X$-evolution that follows. If the time parameter for $X$-evolution is less than $\pi$, that will be the last factor in the decomposition, as the control switch can not occur. Otherwise, we get $R (t W_+)$ followed by a factor $R (\pi X)$. But this will imply optimality of the expression
$R (t W_+) R (-\pi X)$, which gives a contradiction since $R (t W_+) R (-\ep X)$
can not be optimal since it does not satisfy the necessary conditions for optimality of Theorem \ref{PMP}. All other cases are analogous and we conclude that in case $c=0$ factors  $R( t W_\pm)$ in optimal decompositions may be preceded or followed by just a single
factor $R ( t^\prime X)$ or   $R ( t^\prime Y)$ with $|t^\prime| < \pi$,
thus completing the proof of Theorem \ref{czero}.

% \begin{equation}
% h_0 R (\pm t_1 W_+) h_1 R (\pm t_2 W_-) h_2 R (\pm t_3 W_+) \ldots
% h_{n-1} R (\pm t_n W_\pm) h_n , 
% \end{equation}
% with $t_k \geq 0$, where the consecutive factors $R (t_k W_+)$ and $R (t_{k+1} % W_-)$ are joined by $h_k = R (\htx X)$, and so forth, producing subwords:
% \begin{multline}
% R (t_k W_+)  R (\htx X) R (t_{k+1} W_-), \
% R (t_k W_+)  R (- \hty Y) R (-t_{k+1} W_-), \
% R (-t_k W_+)  R (-\htx X) R (-t_{k+1} W_-), \
% R (-t_k W_+)  R (\hty Y) R (t_{k+1} W_-), \\
% R (t_k W_-)  R (\htx X) R (t_{k+1} W_+), \
% R (-t_k W_+)  R (- \hty Y) R (t_{k+1} W_+), \
% R (-t_k W_-)  R (-\htx X) R (-t_{k+1} W_+), \
% R (t_k W_-)  R (\hty Y) R (-t_{k+1} W_+) .
% \end{multline} 
% When $c = 0$ we get $\htx = \hty = \pi$,  $W_+ = X - \kappa Y$ and
% $W_- = X + \kappa Y$. Also note that $W_+$ and $W_-$ have equal cost.
% 
% Carrying out computations in $SU(2)$, it is easy to see that
% \begin{multline}
% R ( \pi X) R (t W_-) = R( t W_+) R (\pi X), \
% R (\pi X) R (t W_+) = R( tW_-) R(\pi X), \\
% R( \pi Y) R (t W_-) = R
% \end{multline}

\end{proof}

Finally, it remains to investigate the factors that could precede/follow $R (t W_+)$ in optimal decompositions. We are going to show that 
the number of such factors is at most two.

\begin{proposition}
\label{fiveW}
Let $c > 0$, $0 < \kappa \leq 1$. Suppose an optimal decomposition for $g \in SO(3)$ contains a factor $R (t W_+)$ with $t > 0$. Then there exists an optimal decomposition of $g$, which is a subword in one of the following: 
\begin{align*}
& R (\hty Y) R (\htx X) R(t W_+) R (\htx X) R (\hty Y), \\
& R (\hty Y) R (\htx X) R(t W_+) R (-\hty Y) R (-\htx X), \\
& R (-\htx X) R (-\hty Y) R(t W_+) R (\htx X) R (\htx Y), \\
& R (-\htx X) R (-\hty Y) R(t W_+) R (-\hty Y) R (-\htx X).
\end{align*}
\end{proposition}
\begin{proof}
Let us show that in an optimal decomposition of $g$ the number of factors following $R (t W_+)$ is at most two.
% Let us consider the factors in the optimal decomposition of $g$, following 
% $R (t W_+)$. We are going to show that there is an optimal decomposition of $g$ where the number of such factors does not exceed two. 
Indeed, if it is followed by
three or more factors, such evolution must begin with either $R (\htx X) R (\hty Y)$
or $R (-\hty Y) R (-\htx X)$. By Proposition \ref{Wconj},  
$R (\htx X) R (\hty Y) =R (-\hty Y) R (-\htx X)$. This will be followed by evolution
with control $-X$ or $Y$. This would imply optimality of either
$R(t W_+) R (\htx X) R ((\hty + \delta) Y)$ or 
\break
$R(t W_+) R (-\hty Y) R (-(\htx + \delta) X)$ for small $\delta > 0$. Let us show that these decompositions are not optimal.               
Consider a preimage $\exp (s W_+) \exp (\hsx X) \exp( (\hsy + \epsilon) Y)$ in $SU(2)$ for
% \break
$R(t W_+) R (\htx X) R ((\hty + \delta) Y)$. Here $s = \frac{t}{2} > 0$, 
$\epsilon = \frac{\delta}{2} > 0$. Since $\kappa c > 0$ we get that $\hsy < \frac{\pi}{2}$ and we can assume that  $\hsy + \epsilon < \frac{\pi}{2}$.
Choose $0 < \tau <  \frac{\pi}{2}$ such that $\tan \tau \tan (\hsy + \epsilon) = \frac{1}{c}$. Since $\tan \hsx \tan \hsy = \frac{1}{c}$, we conclude that
$0 < \tau < \htx$. Then by Proposition \ref{flip}(a) we get
\begin{equation*}
\exp (s W_+) \exp (\hsx X) \exp( (\hsy + \epsilon) Y)
= -\exp (s W_+) \exp ((\hsx - \tau) X) \exp(- (\hsy + \epsilon) Y) \exp(- \tau X).
\end{equation*}
However the latter decomposition is not optimal since it does not correspond to a trajectory in Fig. \ref{kgtc}, \ref{kltc}, yet both sides in the above equality have the same cost. This implies that $R(t W_+) R (\htx X) R ((\hty + \delta) Y)$ is not optimal.
The argument for $R(t W_+) R (-\hty Y) R (-(\htx + \delta) X)$ is analogous.
This completes the proof of Proposition \ref{fiveW} and Theorems \ref{Mgt} and \ref{Mlt}.

\end{proof}

\end{document}